\documentclass[11pt]{article}
\usepackage[tbtags]{amsmath}
\usepackage{amssymb}
\usepackage{amsthm}
\usepackage[misc]{ifsym}
\usepackage{cases}
\usepackage{mathrsfs}
\usepackage{color}
\numberwithin{equation}{section}   
\normalsize
\title{\bf Linear Quadratic Optimal Control Problems of Delayed Backward Stochastic Differential Equations
 \thanks{This work is financially supported by the National Key R\&D Program of China (2018YFB1305400), and the National Natural Science Foundations of China (11971266, 11571205, 11831010).}}
\author{\normalsize Weijun Meng\thanks{\it School of Mathematics, Shandong University, Jinan 250100, P.R. China, E-mail: 201611337@mail.sdu.edu.cn}, Jingtao Shi\thanks{\it Corresponding author, School of Mathematics, Shandong University, Jinan 250100, P.R. China, E-mail: shijingtao@sdu.edu.cn}}
\newtheorem{mypro}{Proposition}[section]
\newtheorem{mythm}{Theorem}[section]

\newtheorem{mylem}{Lemma}[section]
\newtheorem{Remark}{Remark}[section]
\begin{document}
\maketitle

\noindent{\bf Abstract:}\quad
This paper is concerned with a linear quadratic optimal control problem of delayed backward stochastic differential equations. An explicit representation is derived for the optimal control, which is a linear feedback of the entire past history and the future state trajectory in a short period of time. This is one of the major distinctive features of the delayed backward stochastic linear quadratic optimal control problem. To obtain the optimal feedback, a new class of delayed Riccati equations is introduced and the unique solvability of their solutions are discussed in detail.
\vspace{2mm}


\vspace{2mm}


\section{Introduction}

The stochastic control problems with delay have attracted more and more scholars' attention in recent years, due to their wide applications in various fields such as economics, engineering, information science, and networked communication. \emph{Stochastic differential delayed equations} (SDDEs, for short) are nice tools to described the dynamics of some natural and social phenomena (see Mohammed \cite{Mohammed84,Mohammed96}). Since then, massive research on related topics has become a desirable and serious endeavor among researchers in stochastic optimal control, differential games, mathematical finance and so on (see \cite{OS00,PY09,CW10,OSZ11,Yu12,CWY12,HLS12,HS12,DHQ13,WS17,ZX17,XSZ18,LWW18}). The forward SDDEs characterize the dynamic changes of state processes with given initial state trajectories. However, in the financial investment problems we usually prefer to study the dynamic changes of state processes with specified terminal states. Pardoux and Peng \cite{PP90} established the general theory of \emph{backward stochastic differential equations} (BSDEs, for short), whose solution is a pair of adapted processes when the terminal state is given. Since the BSDE itself is a nice dynamic structure, there have been abundant research results about the optimal control problems and differential games of BSDEs, for example \cite{Peng93,DZ99,LZ01,YJ08,HWX09,WY10,WY12,WXX18,LSX19}.

Recently, Delong and Imkeller \cite{DI10} introduced BSDEs with time delayed generators, which are generalization of BSDEs by adding the influence of time delay. Further, Delong \cite{Delong12} studied their applications in finance and insurance. Due to these, it is necessary and urgent to study the optimal control problems of BSDEs with time delayed generators. This is an interesting but challenging topic owing to the influence of the time delay and the backward structure on the controlled systems. A few research can be found in this kind of optimal control problems and their applications. Shi \cite{Shi11} considered the optimal control problem described by a kind of BSDE with time delayed generator and proved a sufficient maximum principle. Shi \cite{Shi13} generalized the above problem to the case driven by Brownian motion and Poisson random measure, by introducing a new class of \emph{time-advanced stochastic differential equations} (ASDEs for short) with jumps as the adjoint equation, and gave the sufficient maximum principle. Chen and Huang \cite{CH15} investigated a stochastic recursive delayed control problem and derived the necessary and sufficient conditions of the maximum principle, by introducing a kind of more general ASDEs as the adjoint equation. Wu and Wang \cite{WW15} focused on the optimal control problem of delayed BSDEs under partial information and the necessary and sufficient conditions of optimality are obtained. Shi and Wang \cite{SW16} studied a nonzero sum differential game of BSDEs with time-delayed generator and gave an Arrow's sufficient condition for the open-loop equilibrium point.

As one of the important special cases of optimal control problems, the linear-quadratic (LQ for short) optimal control problems have been a hot topic for a long period. However, to our knowledge, the literatures about LQ optimal control problems of delayed BSDEs are very scarce. Although the above literatures have discussed the LQ cases, either the state feedback expressions of the optimal controls are not given, or the controlled systems are very special. Hence this paper aims to study the general LQ optimal control problem of delayed BSDE, which we called {\it delayed backward stochastic linear-quadratic} (D-BSLQ for short) optimal control problem. The main contributions of this paper can be summarized in three aspects.
\begin{itemize}
\item Firstly, a general D-BSLQ optimal control problem is proposed and solved completely by the completion-of-squares technique. The optimal control is expressed as a linear feedback of the entire past history and the future state trajectory in a short period of time, which is different from the BSLQ optimal control problem without delay (Lim and Zhou \cite{LZ01}). Furthermore, the optimal cost is expressed by a delayed Riccati equation and a {\it delayed and time-advanced BSDE} (DABSDE for short). See Theorems 3.1, 3.2 in Section 3.
\item Secondly, it is interesting that a new class of {\it time-advanced SDDEs} (ASDDEs for short) is introduced to seek the state feedback expression of the optimal control, which has not been studied yet although it has considerable study value. See (\ref{linear ASDDE}) in Section 3.
\item Thirdly, the delayed Riccati equations mentioned above play a very important role in our analysis. Moreover, to the best of our knowledge, this class of delayed Riccati equations have not appeared in the previous literature. Thus, the existence and uniqueness of their solutions are discussed in detail in Section 4.
\end{itemize}

This paper is organized as follows. In Section 2, the D-BSLQ optimal control problem is formulated and some preliminary results on ASDDEs are given. In Section 3, the main results about the state feedback expression of the optimal control in this paper, are presented. Section 4 is devoted to the existence and uniqueness of solutions of certain Riccati equations. Finally the proofs of the main results are carried out in Section 5.

\section{Preliminaries}

Throughout the paper, $\mathbf{R}^{n\times m}$ is the Euclidean space of all $n\times m$ real matrices, $\mathbf{S}^n$ is the space of all $n\times n$ symmetric matrices, $\mathbf{S}^n_+$ is the subset of $\mathbf{S}^n$ consisting of positive semi-definite matrices, $\bar{\mathbf{S}}^n_+$ is the subset of $\mathbf{S}^n$ consisting of positive definite matrices. We simply write $\mathbf{R}^{n\times m}$ as $\mathbf{R}^n$ when $m=1$. The norm in $\mathbf{R}^n$ is denoted by $|\cdot|$ and the inner product is denoted by $\langle\cdot,\cdot\rangle$. The transpose of vectors or matrices is denoted by the superscript $^\top$. $\mathbf{R}^+=[0,\infty)$ and $\mathbf{N}^+$ is the set of all positive integers. $I$ is the identity matrix with appropriate dimension.

Suppose that $(\Omega,\mathcal{F},\{\mathcal{F}_t\},\mathbb{P})$ is a complete filtered probability space, $\{\mathcal{F}_t\}_{t\geq 0}$ is generated by the one-dimensional standard Brownian motion $\{W(t)\}_{t\geq0}$, and $\mathbb{E}$ denotes the mathematical expectation with respect to the probability $\mathbb{P}$. Let $T>0$ be the finite time duration and $\delta>0$ be a sufficiently small time delay parameter.

First we define the following spaces which will be used in this paper:
\begin{eqnarray*}\begin{aligned}
   &L^p([0,T];\mathbf{R}^{n\times n}):=\bigg\{\mathbf{R}^{n\times n}\mbox{-valued funciton }\phi(t);\ \int_0^T|\phi(t)|^pdt<\infty\bigg\},\\
   &L^{\infty}([0,T];\mathbf{R}^{n\times n}):=\bigg\{\mathbf{R}^{n\times n}\mbox{-valued funciton }\phi(t);\ \sup\limits_{0\leq t\leq T}|\phi(t)|dt<\infty\bigg\},\\
   &L^2_{\mathcal{F}_T}(\mathbf{R}^n):=\bigg\{\mathbf{R}^n\mbox{-valued }\mathcal{F}_T\mbox{-measurable random variable } \xi;\ \mathbb{E}|\xi|^2<\infty\bigg\},\\
   &C([0,T];\mathbf{R}^n):=\bigg\{\mathbf{R}^n\mbox{-valued continuous funciton }\phi(t);\ \sup\limits_{0\leq t\leq T}|\phi(t)|<\infty\bigg\},\\
   &L^2_\mathcal{F}([0,T];\mathbf{R}^n):=\bigg\{\mathbf{R}^n\mbox{-valued }\mathcal{F}_t\mbox{-adapted process }\phi(t);\ \mathbb{E}\int_0^T|\phi(t)|^2dt<\infty\bigg\},\\
   &L^2_\mathcal{F}(\Omega;C([0,T];\mathbf{R}^n)):=\bigg\{\mathbf{R}^n\mbox{-valued }\mathcal{F}_t\mbox{-adapted process }\phi(t);\
    \mathbb{E}\Big[\sup\limits_{0\leq t\leq T}|\phi(t)|^2\Big]<\infty\bigg\}.
\end{aligned}\end{eqnarray*}

\vspace{1mm}

Next we formulate the optimal control problem which will be studied in this paper. For given $s\in[0,T]$, let us consider the following controlled linear delayed BSDE:
\begin{eqnarray}\left\{\begin{aligned}\label{controlled delayed BSDE}
  -dY(t)=&\ \big[A(t)Y(t)+\bar{A}(t)Y(t-\delta)+B(t)Z(t)+\bar{B}(t)Z(t-\delta)+C(t)u(t)\\
         &+\bar{C}(t)u(t-\delta)\big]dt-Z(t)dW(t),\quad t\in[s,T],\\
    Y(T)=&\ \xi,\ Y(t)=\varphi(t),\ Z(t)=\psi(t),\ u(t)=\eta(t),\quad t\in[s-\delta,s),
\end{aligned}\right.\end{eqnarray}
along with the cost functional
\begin{eqnarray}\begin{aligned}\label{cost functional}
    J(s,\xi;u(\cdot))=&\ \mathbb{E}\bigg\{\big\langle\bar{G}Y(s-\delta),Y(s-\delta)\big\rangle+\big\langle GY(s),Y(s)\big\rangle+\int_s^T\Big[\big\langle Q(t)Y(t),Y(t)\big\rangle\\
                      &\quad+\big\langle\bar{Q}(t)Y(t-\delta),Y(t-\delta)\big\rangle+\big\langle R(t)Z(t),Z(t)\big\rangle+\big\langle\bar{R}(t)Z(t-\delta),Z(t-\delta)\big\rangle\\
                      &\quad+\big\langle N(t)u(t),u(t)\big\rangle+\big\langle\bar{N}(t)u(t-\delta),u(t-\delta)\big\rangle\Big]dt\bigg\},
\end{aligned}\end{eqnarray}
where $A(\cdot),\bar{A}(\cdot),B(\cdot),\bar{B}(\cdot),C(\cdot),\bar{C}(\cdot)$ are deterministic matrix-valued functions, $\xi$ is an $\mathcal{F}_T$-measurable random vector, $G,\bar{G}$ are symmetric matrices, $Q(\cdot),\bar{Q}(\cdot),R(\cdot),\bar{R}(\cdot),N(\cdot),\bar{N}(\cdot)$ are deterministic matrix-valued functions, with appropriate dimensions.

The admissible control set is defined as follows:
$$\mathcal{U}[s,T]:=\Big\{u:[s,T]\times\Omega\rightarrow\mathbf{R}^d\big|u(\cdot) \mbox{ is an }\mathcal{F}_t\mbox{-predictable process};\ \mathbb{E}\int_s^T|u(t)|^2dt<\infty. \Big\}.$$

The LQ optimal control problem of delayed BSDE, which we call {\it delayed backward stochastic LQ optimal control problem}, can be stated as follows:

\textbf{Problem (D-BSLQ).} For any $s\in[0,T]$, $\xi\in L^2_{\mathcal{F}_T}(\mathbf{R}^n)$, to find a $u^*(\cdot)\in\mathcal{U}[s,T]$ such that
\begin{equation}\label{D-BSLQ}
  J(s,\xi;u^*(\cdot))=\inf\limits_{u(\cdot)\in\mathcal{U}[s,T]}J(s,\xi;u(\cdot)):=V(s,\xi).
\end{equation}
Any $u^*(\cdot)\in \mathcal{U}[s,T]$ that achieves the above infimum is called an {\it optimal control} and the corresponding solution $(Y^*(\cdot),Z^*(\cdot))$ is called the {\it optimal state trajectory}. The function $V(\cdot,\cdot)$ is called the {\it value function} of \textbf{Problem (D-BSLQ)}.

\vspace{1mm}

Now we introduce the following assumptions that will be in force throughout the paper.
\textbf{(A1)} The coefficients of the state equation (\ref{controlled delayed BSDE}) satisfy the following assumptions:
\begin{equation*}
  A(\cdot),\bar{A}(\cdot),B(\cdot),\bar{B}(\cdot)\in L^\infty([0,T];\mathbf{R}^{n\times n}),\
  C(\cdot),\bar{C}(\cdot)\in L^\infty([0,T];\mathbf{R}^{n\times d}).
\end{equation*}
\textbf{(A2)} The initial trajectory of the state equation (\ref{controlled delayed BSDE}) satisfies $\varphi(\cdot),\psi(\cdot),\eta(\cdot)\in L^2_{\mathcal{F}}([s-\delta,s];\mathbf{R}^n)$.
\textbf{(A3)} The weight coefficients of the cost functional (\ref{cost functional}) satisfy the following assumptions:
\begin{equation*}\left\{\begin{aligned}
  &G,\bar{G}\in \mathbf{S}^n,\quad Q(\cdot),\bar{Q}(\cdot)\in L^\infty([0,T];\mathbf{S}^{n}),\\
  &R(\cdot),\bar{R}(\cdot)\in L^\infty([0,T];\mathbf{S}^n),\quad N(\cdot),\bar{N}(\cdot)\in L^\infty([0,T];\mathbf{S}^d),\\
\end{aligned}\right.\end{equation*}
and there exists a constant $\alpha>0$ such that
\begin{equation*}
  G\geq 0,\ Q(t)+\bar{Q}(t+\delta)\geq 0,\ R(t)+\bar{R}(t+\delta)\geq 0,\ N(t)+\bar{N}(t+\delta)\geq \alpha I,\ a.e.\ t\in[s-\delta,T].
\end{equation*}

\vspace{1mm}

Next we give a result to guarantee the well-posedness of the state equation (\ref{controlled delayed BSDE}).

\begin{mythm}\label{thm2.1}
Let \textbf{(A1)}, \textbf{(A2)} hold. Then for any $(\xi,u(\cdot))\in L^2_{\mathcal{F}_T}(\mathbf{R}^n)\times\mathcal{U}[s,T]$, the state equation (\ref{controlled delayed BSDE}) admits a unique adapted solution $(Y(\cdot),Z(\cdot))\in L^2_{\mathcal{F}}(\Omega;C([s,T];\mathbf{R}^n))\times L^2_{\mathcal{F}}([s,T];\mathbf{R}^n).$ Moreover, there exists a constant $K>0$, independent of $\xi$ and $u(\cdot)$ such that
\begin{equation}\label{estimate}\begin{aligned}
  &\mathbb{E}\bigg[\sup\limits_{s\leq t\leq T}|Y(t)|^2+\int_s^T|Z(t)|^2dt\bigg]\\
  &\leq K\mathbb{E}\bigg[|\xi|^2+\int_s^T|u(t)|^2dt+\int_{s-\delta}^{s}\big(|\varphi(t)|^2+|\psi(t)|^2+|\eta(t)|^2\big)dt\bigg].
\end{aligned}\end{equation}
\end{mythm}

\begin{proof}
We try to use the contraction mapping theorem to prove the result. For any $\beta\in\mathbf{R}$, let
\begin{equation*}
  \mathcal{M}_\beta[s,T]:=L_\mathcal{F}^2(\Omega;C([s,T];\mathbf{R}^n))\times L^2_{F}([s,T];\mathbf{R}^n),
\end{equation*}
equipped with the norm
\begin{equation*}
   \big|\big|(Y,Z)\big|\big|_{\mathcal{M}_\beta[s,T]}:=\bigg(\mathbb{E}\Big[\sup\limits_{s\leq t\leq T}|Y(t)|^2e^{\beta h(t)}\Big]
   +\mathbb{E}\int_s^T|Z(t)|^2e^{\beta h(t)}dt\bigg)^{\frac{1}{2}},
\end{equation*}
where $h(t):=\int_s^t\big[|A(r)|+|\bar{A}(r)|+|B(r)|^2+|\bar{B}(r)|^2\big]dr,\ t\in[s,T].$ We define the mapping $\mathcal{T}:\mathcal{M}_{\beta}[s,T]\rightarrow\mathcal{M}_{\beta}[s,T]$, apparently it is well-defined. In fact, for any $(y(\cdot),z(\cdot))\in\mathcal{M}_\beta[s,T]$ and $y(t)=\varphi(t),z(t)=\psi(t),\ s-\delta\leq t<s$, consider the following BSDE:
\begin{eqnarray*}\left\{\begin{aligned}
  -dY(t)=&\ \big[A(t)y(t)+\bar{A}(t)y(t-\delta)+B(t)z(t)+\bar{B}(t)z(t-\delta)+C(t)u(t)\\
         &\ +\bar{C}(t)u(t-\delta)\big]dt-Z(t)dW(t),\quad t\in[s,T],\\
    Y(T)=&\ \xi,\ Y(t)=\varphi(t),\ Z(t)=\psi(t),\ u(t)=\eta(t),\quad t\in[s-\delta,s),
\end{aligned}\right.\end{eqnarray*}
then by Pardoux and Peng \cite{PP90}, $(Y(\cdot),Z(\cdot))=\mathcal{T}(y(\cdot),z(\cdot))\in\mathcal{M}_{\beta}[s,T]$.

For any $(y_1(\cdot),z_1(\cdot)),(y_2(\cdot),z_2(\cdot))\in\mathcal{M}_{\beta}[s,T],$ and $y_1(t)=y_2(t)=\varphi(t),$ $z_1(t)=z_2(t)=\psi(t)$, $t\in[s-\delta,s)$, denote
\begin{equation*}\left\{\begin{aligned}
  &(Y_1(\cdot),Z_1(\cdot))=\mathcal{T}(y_1(\cdot),z_1(\cdot)),\quad (Y_2(\cdot),Z_2(\cdot))=\mathcal{T}(y_2(\cdot),z_2(\cdot)),\\
  &\hat{Y}(\cdot)=Y_1(\cdot)-Y_2(\cdot),\quad \hat{Z}(\cdot)=Z_1(\cdot)-Z_2(\cdot),\quad
  \hat{y}(\cdot)=y_1(\cdot)-y_2(\cdot),\quad \hat{z}(\cdot)=z_1(\cdot)-z_2(\cdot).
\end{aligned}\right.\end{equation*}
Then we have
\begin{equation*}\left\{\begin{aligned}
  -d\hat{Y}(t)&=\big[A(t)\hat{y}(t)+\bar{A}(t)\hat{y}(t-\delta)+B(t)\hat{z}(t)+\bar{B}(t)\hat{z}(t-\delta)\big]dt\\
              &\quad-\hat{Z}(t)dW(t),\quad t\in[s,T],\\
    \hat{Y}(T)&=0,\ \hat{Y}(t)=0,\ \hat{Z}(t)=0,\ t\in[s-\delta,s).
\end{aligned}\right.\end{equation*}
Applying It\^o's formula to $e^{\beta h(t)}|\hat{Y}(t)|^2$, we obtain
\begin{equation*}\begin{aligned}
  -e^{\beta h(t)}|\hat{Y}(t)|^2=&\int_t^Te^{\beta h(r)}\Big[\beta|\hat{Y}(r)|^2\big(|A(r)|+|\bar{A}(r)|+|B(r)|^2+|\bar{B}(r)|^2\big)
   -2\big\langle\hat{Y}(r),A(r)\hat{y}(r)\big\rangle\\
  &\ -2\big\langle\hat{Y}(r),\bar{A}(r)\hat{y}(r-\delta)\big\rangle-2\big\langle\hat{Y}(r),B(r)\hat{z}(r)\big\rangle-2\big\langle\hat{Y}(r),\bar{B}(r)\hat{z}(r-\delta)\big\rangle\\
  &\ +|\hat{Z}(r)|^2\Big]dr+2\int_t^Te^{\beta h(r)}\big\langle\hat{Y}(r),\hat{Z}(r)\big\rangle dW(r).
\end{aligned}\end{equation*}
Then we deduce
\begin{eqnarray}\begin{aligned}\label{eq2.5}
     &e^{\beta h(t)}|\hat{Y}(t)|^2+\int_t^T\beta e^{\beta h(r)}|\hat{Y}(r)|^2\big(|A(r)|+|\bar{A}(r)|+|B(r)|^2+|\bar{B}(r)|^2\big)dr\\
     &+\int_t^Te^{\beta h(r)}|\hat{Z}(r)|^2dr+2\int_t^Te^{\beta h(r)}\big\langle\hat{Y}(r),\hat{Z}(r)\big\rangle dW(r)\\
\leq &\int_t^T\beta e^{\beta h(r)}\big(|A(r)|+|\bar{A}(r)|+|B(r)|^2+|\bar{B}(r)|^2\big)|\hat{Y}(r)|^2dr\\
     &+\beta^{-1}\int_t^Te^{\beta h(r)}\Big[|A(r)||\hat{y}(r)|^2+|\bar{A}(r)||\hat{y}(r-\delta)|^2+|\hat{z}(r)|^2+|\hat{z}(r-\delta)|^2\Big]dr\\
\leq &\int_t^T\beta e^{\beta h(r)}\big(|A(r)|+|\bar{A}(r)|+|B(r)|^2+|\bar{B}(r)|^2\big)|\hat{Y}(r)|^2dr\\
     &+\beta^{-1}\int_t^Te^{\beta h(r)}\big(|A(r)|+e^{\beta(h(r+\delta)-h(r))}|\bar{A}(r+\delta)|\big)|\hat{y}(r)|^2dr
      +\beta^{-1}\int_{t-\delta}^te^{\beta h(r+\delta)}|\hat{z}(r)|^2dr\\
     &+\beta^{-1}\int_{t-\delta}^te^{\beta h(r+\delta)}|\bar{A}(r+\delta)||\hat{y}(r)|^2dr
      +\beta^{-1}\int_t^Te^{\beta h(r)}\big(1+e^{\beta(h(r+\delta)-h(r))}\big)|\hat{z}(r)|^2dr.
\end{aligned}\end{eqnarray}
And thus we derive
\begin{equation}\begin{aligned}\label{eq2.6}
     &\mathbb{E}\int_s^Te^{\beta h(r)}|\hat{Z}(r)|^2dr
\leq \beta^{-1}\mathbb{E}\int_s^Te^{\beta h(r)}\big(|A(r)|+e^{\beta[h(r+\delta)-h(r)]}|\bar{A}(r+\delta)|\big)|\hat{y}(r)|^2dr\\
     &+\beta^{-1}\mathbb{E}\int_s^Te^{\beta h(r)}\big(1+e^{\beta[h(r+\delta)-h(r)]}\big)|\hat{z}(r)|^2dr\\
\leq &\ \beta^{-1}\mathbb{E}\bigg[\sup\limits_{s\leq r\leq T}e^{\beta h(r)}|\hat{y}(r)|^2\int_s^T\big(|A(r)|+e^{\beta[h(r+\delta)-h(r)]}|\bar{A}(r+\delta)|\big)dr\bigg]\\
     &+\beta^{-1}\Big(1+\sup\limits_{s\leq r\leq T}e^{\beta[h(r+\delta)-h(r)]}\Big)\mathbb{E}\int_s^Te^{\beta h(r)}|\hat{z}(r)|^2dr
\leq \tilde{\beta}\big|\big|(\hat{y},\hat{z})\big|\big|^2_{\mathcal{M}_\beta[s,T]},
\end{aligned}\end{equation}
where $\tilde{\beta}:=\beta^{-1}\big(1+\sup\limits_{s\leq r\leq T}e^{\beta[h(r+\delta)-h(r)]}\big)\big[1+\int_s^T\big(|A(r)|+|\bar{A}(r+\delta)|\big)dr\big]$.

On the other hand, by Burkholder-Davis-Gundy's inequality, we have
\begin{equation}\begin{aligned}\label{eq2.7}
     &\mathbb{E}\bigg[\sup\limits_{s\leq t\leq T}\bigg|\int_t^Te^{\beta h(r)}\big\langle\hat{Y}(r),\hat{Z}(r)\big\rangle dW(r)\bigg|\bigg]
\leq C\mathbb{E}\bigg(\int_s^Te^{2\beta h(r)}|\hat{Y}(r)|^2|\hat{Z}(r)|^2dr\bigg)^\frac{1}{2}\\
\leq &\ C\mathbb{E}\bigg[\sup\limits_{s\leq t\leq T}e^{\frac{1}{2}\beta h(r)}|\hat{Y}(r)|\bigg(\int_s^Te^{\beta h(r)}|\hat{Z}(r)|^2dr\bigg)^\frac{1}{2}\bigg]\\
\leq &\ \frac{1}{4}\mathbb{E}\bigg[\sup\limits_{s\leq t\leq T}e^{\beta h(r)}|\hat{Y}(r)|^2\bigg]+C^2\mathbb{E}\int_s^Te^{\beta h(r)}|\hat{Z}(r)|^2dr,
\end{aligned}\end{equation}
where $C$ is a constant. Combining (\ref{eq2.5}), (\ref{eq2.6}) and (\ref{eq2.7}), we obtain
\begin{equation}\begin{aligned}\label{eq2.8}
     &\mathbb{E}\bigg[\sup\limits_{s\leq t\leq T}e^{\beta h(t)}|\hat{Y}(t)|^2\bigg]
\leq 2\tilde{\beta}\big|\big|(\hat{y},\hat{z})\big|\big|^2_{\mathcal{M}_\beta[s,T]}+4C^2\mathbb{E}\int_s^Te^{\beta h(r)}|\hat{Z}(r)|^2dr\\
     &+2\beta^{-1}\mathbb{E}\bigg[\sup\limits_{s\leq t\leq T}\int_{t-\delta}^te^{\beta h(r+\delta)}|\bar{A}(r+\delta)||\hat{y}(r)|^2dr\bigg]
      +2\beta^{-1}\mathbb{E}\bigg[\sup\limits_{s\leq t\leq T}\int_{t-\delta}^te^{\beta h(r+\delta)}|\hat{z}(r)|^2dr\bigg]\\
\leq &\ 6\tilde{\beta}\big|\big|(\hat{y},\hat{z})\big|\big|^2_{\mathcal{M}_\beta[s,T]}+4C^2\mathbb{E}\int_s^Te^{\beta h(r)}|\hat{Z}(r)|^2dr.
\end{aligned}\end{equation}
Finally, by (\ref{eq2.6}) and (\ref{eq2.8}), we deduce
\begin{equation}\begin{aligned}\label{eq2.9}
\big|\big|(\hat{Y},\hat{Z})\big|\big|^2_{\mathcal{M}_\beta[s,T]}\leq\tilde{\beta}(7+4C^2)\big|\big|(\hat{y},\hat{z})\big|\big|^2_{\mathcal{M}_\beta[s,T]}.
\end{aligned}\end{equation}
Thus as long as $\tilde{\beta}(7+4C^2)<1$, then $\mathcal{T}$ is a contraction mapping. Since \textbf{(A1)} holds and $\delta$ is sufficiently small, we can indeed choose some sufficiently large $\beta$ such that (\ref{eq2.9}) holds.

As for the estimate (\ref{estimate}), let $(Y_0(\cdot),Z_0(\cdot))=\mathcal{T}(0,0),$ then
\begin{equation*}\begin{aligned}
  Y_0(t)=&\ \xi+\int_{t-\delta}^s\bar{A}(r+\delta)\varphi(r)dr+\int_{t-\delta}^s\bar{B}(r+\delta)\psi(r)dr\\
         &+\int_{t}^T[C(r)u(r)+\bar{C}(r)u(r-\delta)]dr-\int_t^TZ_0(r)dW(r).
\end{aligned}\end{equation*}
Similarly, for some $\tilde{\beta}^\prime\in(0,1)$, we have
\begin{equation*}
  \big|\big|(\hat{Y},\hat{Z})-(\hat{Y_0},\hat{Z_0})\big|\big|^2_{\mathcal{M}_\beta[s,T]}\leq\tilde{\beta}^\prime\big|\big|(\hat{Y},\hat{Z})\big|\big|^2_{\mathcal{M}_\beta[s,T]}.
\end{equation*}
Hence we get
\begin{equation*}
  \big|\big|(\hat{Y},\hat{Z})\big|\big|^2_{\mathcal{M}_\beta[s,T]}\leq K\big|\big|(\hat{Y_0},\hat{Z_0})\big|\big|^2_{\mathcal{M}_\beta[s,T]}.
\end{equation*}
It follows that
\begin{equation*}\begin{aligned}
     &\big|\big|(\hat{Y},\hat{Z})\big|\big|^2_{\mathcal{M}_0[s,T]}
      \leq K\big|\big|(\hat{Y_0},\hat{Z_0})\big|\big|^2_{\mathcal{M}_\beta[s,T]}\leq K\big|\big|(\hat{Y_0},\hat{Z_0})\big|\big|^2_{\mathcal{M}_0[s,T]}\\
\leq &\ K\mathbb{E}\bigg[|\xi|^2+\bigg(\int_{s-\delta}^s|\bar{A}(r+\delta)\varphi(r)|dr+\int_{s-\delta}^s|\bar{B}(r+\delta)\psi(r)|dr\\
     &\qquad+\int_s^T|C(r)u(r)+\bar{C}(r)u(r-\delta)|dr\bigg)^2\bigg]\\
\leq &\ K\mathbb{E}\bigg[|\xi|^2+\int_s^T|u(t)|^2dt+\int_{s-\delta}^{s}\big(|\varphi(t)|^2+|\psi(t)|^2+|\eta(t)|^2\big)dt\bigg],
\end{aligned}\end{equation*}
where $K>0$ is a generic constant. The proof is complete.
\end{proof}

\begin{Remark}\label{re2.1}
From Theorem \ref{thm2.1}, we can easily know that under \textbf{(A1)-(A2)}, for any $(s,\xi)\in[0,T]\times L^2_{\mathcal{F}_T}(\mathbf{R}^n)$ and $u(\cdot)\in\mathcal{U}[s,T]$, the cost functional (\ref{cost functional}) is well-posed and hence \textbf{Problem (D-BSLQ)} makes sense.
\end{Remark}

\begin{Remark}\label{re2.2}
From the proof of Theorem \ref{thm2.1}, the conditions imposed on the coefficients of (\ref{controlled delayed BSDE}) can be relaxed. For example, when $ A(\cdot),\bar{A}(\cdot)\in L^1([0,T];\mathbf{R}^{n\times n})$, $B(\cdot),\bar{B}(\cdot)\in L^2([0,T];\mathbf{R}^{n\times n})$, $C(\cdot),\bar{C}(\cdot)\in L^2([0,T];\mathbf{R}^{n\times d})$, (\ref{controlled delayed BSDE}) is still well-posed.
\end{Remark}

In the last part of this section we introduce a new class of ASDDEs.

For any given $s\in[0,T]$, consider the following ASDDE:
\begin{eqnarray}\left\{\begin{aligned}\label{ASDDE}
   dX(t)=&\ b(t,X(t),X(t-\delta_1(t)),X(t+\delta_2(t)))dt\\
         &+\sigma(t,X(t),X(t-\delta_1(t)),X(t+\delta_2(t)))dW(t),\quad t\in [s,T],\\
    X(s)=&\ \zeta,\ X(t)=\alpha_1(t),\quad t\in[s-K_1,s),\ X(t)=\alpha_2(t),\quad t\in(T,T+K_2],
\end{aligned}\right.\end{eqnarray}
where $\delta_1(\cdot),\delta_2(\cdot)$ are $\mathbf{R}^+$-valued continuous functions defined on $[s,T]$ and $s-K_1\leq t-\delta_1(t)\leq T$, $s\leq t+\delta_2(t)\leq K_2+T$ for all $t\in[s,T]$. Moreover, we suppose the following assumptions hold.

(\textbf{H1}) There exists a constant $L>0$ such that
\begin{equation*}\begin{aligned}
&\int_s^Tg(t-\delta_1(t))dt\leq L\int_{s-K_1}^{T}g(t)dt,\\
&\int_s^Tg(t+\delta_2(t))dt\leq L\int_s^{T+K_2}g(t)dt,
\end{aligned}\end{equation*}
for any nonnegative integrable function $g(\cdot)$.

(\textbf{H2}) Let $b(t,\omega,x,\theta,\vartheta):[0,T]\times\Omega\times\mathbf{R}^n\times L^2_{\mathcal{F}_{r_1}}(\mathbf{R}^n)\times L^2_{\mathcal{F}_{r_2}}(\mathbf{R}^n)\rightarrow L^2_{\mathcal{F}_t}(\mathbf{R}^n),\sigma(t,\omega,x,\theta,\vartheta):[0,T]\times\Omega\times\mathbf{R}^n\times L^2_{\mathcal{F}_{r_1}}(\mathbf{R}^n)\times L^2_{\mathcal{F}_{r_2}}(\mathbf{R}^n)\rightarrow L^2_{\mathcal{F}_t}(\mathbf{R}^n)$, where $r_1\in[t-K_1,T]$, $r_2\in[t,T+K_2]$,  satisfy:
\begin{equation*}\begin{aligned}
&\big|b(t,x,\theta_{r_1},\vartheta_{r_2})-b(t,\bar{x},\bar{\theta}_{r_1},\bar{\vartheta}_{r_2})\big|
+\big|\sigma(t,x,\theta_{r_1},\vartheta_{r_2})-\sigma(t,\bar{x},\bar{\theta}_{r_1},\bar{\vartheta}_{r_2})\big|\\
&\leq L\big(|x-\bar{x}|+|\theta_{r_1}-\bar{\theta}_{r_1}|+\mathbb{E}^{\mathcal{F}_t}\big[|\vartheta_{r_2}-\bar{\vartheta_{r_2}}|\big]\big),
\end{aligned}\end{equation*}
for any $t\in[s,T],x,\bar{x}\in\mathbf{R}^n,\theta,\bar{\theta}\in L_{\mathcal{F}}^2([t-K_1,T];\mathbf{R}^n),\vartheta,\bar{\vartheta}\in L_{\mathcal{F}}^2([t,T+K_2];\mathbf{R}^n)$, and for the above constant $L>0$;

$(\textbf{H3})\ \mathbb{E}\int_s^T\big(|b(t,0,0,0)|^2+|\sigma(t,0,0,0)|^2\big)dt<+\infty$.

\begin{Remark}\label{re2.3}
The fact that $b(t,\cdot,\cdot,\cdot)$, $\sigma(t,\cdot,\cdot,\cdot)$ are $\mathcal{F}_t$-measurable guarantees the adaptability of the solution to the ASDDE (\ref{ASDDE}).
\end{Remark}

Then we have the following result.
\begin{mythm}\label{thm2.2}
Let (\textbf{H1})-(\textbf{H3}) hold, suppose $K_2>0$ is sufficiently small, then for any given initial and terminal conditions $\alpha_1(\cdot)\in L^2_{\mathcal{F}}([s-K_1,s];\mathbf{R}^n)$, $\alpha_2(\cdot)\in L^2_{\mathcal{F}}([T,T+K_2];\mathbf{R}^n)$, and for any $\zeta\in L^2_{\mathcal{F}_s}(\mathbf{R}^n)$, the ASDDE (\ref{ASDDE}) has a unique solution $X(\cdot)\in L^2_{\mathcal{F}}(\Omega;C([s,T];\mathbf{R}^n))$.
\end{mythm}

\begin{proof}
For any $\beta>0$, define
\begin{equation*}
  \mathcal{H}_{\beta}[s,T]:=L_\mathcal{F}^2([s,T];\mathbf{R}^n),
\end{equation*}
equipped with the norm
\begin{equation*}
   \big|\big|X\big|\big|_{\mathcal{H}_\beta[s,T]}:=\bigg(\mathbb{E}\int_s^T|X(t)|^2e^{-\beta t}dt\bigg)^\frac{1}{2}.
\end{equation*}
Let the mapping $\mathcal{T}:\mathcal{H}_\beta[s,T]\rightarrow\mathcal{H}_{\beta}[s,T]$, apparently it is well-defined. In fact, for any $x(\cdot)\in\mathcal{H}_{\beta}[s,T]$ and $x(t)=\alpha_1(t),\ s-K_1\leq t< s$, $x(t)=\alpha_2(t),\ T<t\leq T+K_2$, consider the following SDE:
\begin{equation*}\left\{\begin{aligned}
   dX(t)=&\ b(t,X(t),x(t-\delta_1(t)),x(t+\delta_2(t))dt\\
         &+\sigma(t,X(t),x(t-\delta_1(t)),x(t+\delta_2(t))dW(t),\quad t\in [s,T],\\
    X(s)=&\ \zeta,\ X(t)=\alpha_1(t),\quad t\in[s-K_1,s),\ X(t)=\alpha_2(t),\quad t\in(T,T+K_2].
\end{aligned}\right.\end{equation*}
Then $X(\cdot)=\mathcal{T}(x(\cdot))\in\mathcal{H}_\beta[s,T]$ by the existence and uniqueness result of the solution to SDE in Yong and Zhou \cite{YZ99}.
Next for any $x_1(\cdot),x_2(\cdot)\in\mathcal{H}_\beta[s,T]$, and $x_1(t)=x_2(t)=\alpha_1(t),\ t\in[s-K_1,s)$, $x_1(t)=x_2(t)=\alpha_2(t),\ t\in(T,T+K_2]$, denote
\begin{equation*}
   (X_1(\cdot),X_2(\cdot))=(\mathcal{T}(x_1(\cdot)),\mathcal{T}(x_2(\cdot))),\quad \hat{X}(\cdot)=X_1(\cdot)-X_2(\cdot),\quad \hat{x}(\cdot)=x_1(\cdot)-x_2(\cdot).
\end{equation*}
 Then we obtain
\begin{equation*}\left\{\begin{aligned}
   d\hat{X}(t)=&\ \big[b(t,X_1(t),x_1(t-\delta_1(t)),x_1(t+\delta_2(t)))-b(t,X_2(t),x_2(t-\delta_1(t)),x_2(t+\delta_2(t)))\big]dt\\
               &+\big[\sigma(t,X_1(t),x_1(t-\delta_1(t)),x_1(t+\delta_2(t)))\\
               &\quad-\sigma(t,X_2(t),x_2(t-\delta_1(t)),x_2(t+\delta_2(t)))\big]dW(t),\quad t\in [s,T],\\
    \hat{X}(s)=&\ 0,\ \hat{X}(t)=0,\ t\in[s-K_1,s)\cup(T,T+K_2].
\end{aligned}\right.\end{equation*}
 Applying It\^o's formula to $e^{-\beta r}|\hat{X}(r)|^2$, we get
\begin{equation*}\begin{aligned}
   &e^{-\beta r}|\hat{X}(r)|^2=\int_s^re^{-\beta t}\Big[-\beta|\hat{X}(t)|^2+2\big\langle\hat{X}(t),b(t,X_1(t),x_1(t-\delta_1(t)),x_1(t+\delta_2(t)))\\
   &\qquad-b(t,X_2(t),x_2(t-\delta_1(t)),x_2(t+\delta_2(t)))\big\rangle+\big\langle\sigma(t,X_1(t),x_1(t-\delta_1(t)),x_1(t+\delta_2(t)))\\
   &\qquad-\sigma(t,X_2(t),x_2(t-\delta_1(t)),x_2(t+\delta_2(t))),\sigma(t,X_1(t),x_1(t-\delta_1(t)),x_1(t+\delta_2(t)))\\
   &\qquad-\sigma(t,X_2(t),x_2(t-\delta_1(t)),x_2(t+\delta_2(t)))\big\rangle\Big]dt\\
   &\quad+2\int_s^re^{-\beta t}\big\langle\hat{X}(t),\sigma(t,X_1(t),x_1(t-\delta_1(t)),x_1(t+\delta_2(t)))\\
   &\qquad\qquad-\sigma(t,X_2(t),x_2(t-\delta_1(t)),x_2(t+\delta_2(t)))\big\rangle dW(t).
\end{aligned}\end{equation*}
Hence by (\textbf{H2}), we obtain
\begin{eqnarray*}\begin{aligned}
   &\mathbb{E}\int_s^T\beta e^{-\beta t}|\hat{X}(t)|^2dt\\
   &\leq\mathbb{E}\int_s^Te^{-\beta t}\Big[2L|\hat{X}(t)|^2+2L|\hat{X}(t)||\hat{x}(t-\delta_1(t))|+2L|\hat{X}(t)||\hat{x}(t+\delta_2(t))|\\
   &\qquad+3L^2|\hat{X}(t)|^2+3L^2|\hat{x}(t-\delta_1(t))|^2+3L^2|\hat{x}(t+\delta_2(t))|^2\Big]dt\\
   &\leq\mathbb{E}\int_s^Te^{-\beta t}\Big[(2L+5L^2)|\hat{X}(t)|^2+(3L^2+1)|\hat{x}(t-\delta_1(t))|^2+(3L^2+1)|\hat{x}(t+\delta_2(t))|^2\Big]dt.
\end{aligned}\end{eqnarray*}
Noting (\textbf{H1}), we deduce
\begin{equation*}\begin{aligned}
  &\mathbb{E}\int_s^T e^{-\beta t}|\hat{X}(t)|^2dt\\
  &\leq(3L^2+1)\big(\beta-2L-5L^2\big)^{-1}\mathbb{E}\int_s^T e^{-\beta t}\big[|\hat{x}(t-\delta_1(t))|^2+|\hat{x}(t+\delta_2(t))|^2\big]dt\\
  &\leq L(3L^2+1)(1+e^{\beta K_2})\big(\beta-2L-5L^2\big)^{-1}\mathbb{E}\int_s^T e^{-\beta t}|\hat{x}(t)|^2dt.
\end{aligned}\end{equation*}
Since $K_2$ is sufficiently small, we can choose some sufficiently large $\beta$ such that $L(3L^2+1)(1+e^{\beta K_2})(\beta-2L-5L^2)^{-1}<1$, then $\mathcal{T}$ is a contraction mapping. Thus ASDDE (\ref{ASDDE}) has a unique solution $X(\cdot)\in L_{\mathcal{F}}^2([s,T];\mathbf{R}^n)$. Next we try to prove that $X(\cdot)\in L_{\mathcal{F}}^2(\Omega;C([s,T];\mathbf{R}^n))$. In fact,
\begin{equation*}\begin{aligned}
     &\mathbb{E}\Big[\sup\limits_{s\leq t\leq T}|X(t)|^2\Big]\\
\leq &\ 3\mathbb{E}|\zeta|^2+3T\mathbb{E}\int_s^T|b(t,X(t),X(t-\delta_1(t)),X(t+\delta_2(t)))|^2dt\\
     &+3C\mathbb{E}\int_s^T|\sigma(t,X(t),X(t-\delta_1(t)),X(t+\delta_2(t)))|^2dt\\
\leq &\ 3\mathbb{E}|\zeta|^2+12T\mathbb{E}\int_s^T|b(t,0,0,0)|^2dt+12C\mathbb{E}\int_s^T|\sigma(t,0,0,0)|^2dt+12L^2(T+C)\mathbb{E}\int_s^T|X(t)|^2dt\\
     &+12L^2(T+C)\mathbb{E}\int_s^T|X(t-\delta_1(t))|^2dt+12L^2(T+C)\mathbb{E}\int_s^T|X(t+\delta_2(t))|^2dt\\
\leq &\ 3\mathbb{E}|\zeta|^2+12T\mathbb{E}\int_s^T|b(t,0,0,0)|^2dt+12C\mathbb{E}\int_s^T|\sigma(t,0,0,0)|^2dt\\
   &+12L^2(T+C)(1+2L)\mathbb{E}\int_s^T|X(t)|^2dt+12L^2(T+C)\mathbb{E}\int_{s-K_1}^s|\alpha_1(t)|^2dt\\
   &+12L^2(T+C)\mathbb{E}\int_{T}^{T+K_2}|\alpha_2(t)|^2dt<\infty,
\end{aligned}\end{equation*}
where $C>0$ is a constant. Hence we complete the proof.
\end{proof}

\begin{Remark}\label{re2.4}
In fact, ASDDEs are a generalization of ASDEs and SDDEs. Let $\delta_1(\cdot)\equiv 0$, then ASDDE (\ref{ASDDE}) becomes the ASDE. While let $\delta_2(\cdot)\equiv 0$, the ASDDE (\ref{ASDDE}) becomes the SDDE.
\end{Remark}


\section{Representations of Optimal Control and Value Function}

In this section, we present the main results of this paper, which solves the above \textbf{Problem (D-BSLQ)}. We will give two alternative expressions of the solution to the optimal control problem, which in fact, can be proved to be equivalent.

First we introduce the following delayed Riccati equation：
\begin{equation}\left\{\begin{aligned}\label{delayed Riccati}
\dot{\Sigma}(t)=&-\Sigma(t)A(t)^\top-A(t)\Sigma(t)+2\Sigma(t)\big[Q(t)+\bar{Q}(t+\delta)\big]\Sigma(t)\\
                &-B(t)\Sigma(t)\mathcal{M}^{-1}(t)B(t)^\top-C(t)\mathcal{N}^{-1}(t)C(t)^\top\\
                &-\bar{B}(t)\Sigma(t-\delta)\mathcal{M}^{-1}(t-\delta)\bar{B}(t)^\top-\bar{C}(t)\mathcal{N}^{-1}(t-\delta)\bar{C}(t)^\top,\quad t\in[s,T],\\
      \Sigma(T)=&\ 0,\ \Sigma(t)=I,\quad t\in[s-\delta,s),
\end{aligned}\right.\end{equation}
where $\mathcal{M}(t):=2R(t)\Sigma(t)+2\bar{R}(t+\delta)\Sigma(t)+I$, $\mathcal{N}(t):=2N(t)+2\bar{N}(t+\delta)$.
The existence and uniqueness of the solution to (\ref{delayed Riccati}) will be discussed in the next section. Let $\Sigma(\cdot)$ be the solution to (\ref{delayed Riccati}), consider the following equations:
\begin{equation}\left\{\begin{aligned}\label{Riccati without delay}
\dot{L}(t)=&\ L(t)A(t)+A(t)^\top L(t)+2\big[Q(t)+\bar{Q}(t+\delta)\big]\\
           &-L(t)B(t)\Sigma(t)\mathcal{M}^{-1}(t)B(t)^\top L(t)-L(t)\bar{B}(t)\Sigma(t-\delta)\mathcal{M}^{-1}(t-\delta)\bar{B}(t)^\top L(t)\\
           &-L(t)C(t)\mathcal{N}^{-1}(t)C(t)^\top L(t)-L(t)\bar{C}(t)\mathcal{N}^{-1}(t-\delta)\bar{C}(t)^\top L(t),\quad t\in[s,T],\\
      L(s)=&\ 2G,
\end{aligned}\right.\end{equation}
\begin{eqnarray}\left\{\begin{aligned}\label{linear delayed and time-advanced BSDE}
  d\bar{X}(t)=&\Big\{\big[A(t)^\top-2(Q(t)+\bar{Q}(t+\delta))\Sigma(t)\big]\bar{X}(t)+2(Q(t)+\bar{Q}(t+\delta))\Lambda(t)\\
              &+\mathbb{E}^{\mathcal{F}_t}\big[(\bar{A}^\top\bar{X})|_{t+\delta}\big]\Big\}dt+\Big\{\big[B(t)^\top-2[R(t)+\bar{R}(t+\delta)]
               \Sigma(t)\mathcal{M}^{-1}(t)B(t)^\top\big]\bar{X}(t)\\
              &+\mathbb{E}^{\mathcal{F}_t}\big\{\big[\bar{B}(t+\delta)^\top-2[R(t)+\bar{R}(t+\delta)]\Sigma(t)\mathcal{M}^{-1}(t)\bar{B}(t+\delta)^\top\big]
               \bar{X}(t+\delta)\big\}\\
              &\ +2[R(t)+\bar{R}(t+\delta)]\big[2\Sigma(t)(R(t)+\bar{R}(t+\delta))+I\big]^{-1}\Gamma(t)\Big\}dW(t),\\
  d\Lambda(t)=&\Big\{\big\{2\Sigma(t)[Q(t)+\bar{Q}(t+\delta)]-A(t)\big\}\Lambda(t)-B(t)\big\{2\Sigma(t)[R(t)+\bar{R}(t+\delta)]+I\big\}^{-1}\\
              &\times\Gamma(t)-\bar{A}(t)\Lambda(t-\delta)-\bar{B}(t)\big\{2\Sigma(t-\delta)[R(t-\delta)+\bar{R}(t)]+I\}^{-1}\Gamma(t-\delta)\\
              &+\big[\bar{B}(t)\Sigma(t-\delta)\mathcal{M}^{-1}(t-\delta)\bar{B}(t)^\top+\bar{C}(t)\mathcal{N}^{-1}(t-\delta)\bar{C}(t)^\top\big]\\
              &\times\big[\mathbb{E}^{\mathcal{F}_{t-\delta}}[\bar{X}(t)]-\bar{X}(t)\big]\Big\}dt+\Gamma(t)dW(t),\quad t\in[s,T],\\
   \bar{X}(s)=&\ 2G(I+2\Sigma(s)G)^{-1}\Lambda(s),\ \bar{X}(t)=0,\quad t\in(T,T+\delta],\\
   \Lambda(T)=&-\xi,\ \Lambda(t)=0,\ \Gamma(t)=0,\quad t\in[s-\delta,s).
\end{aligned}\right.\end{eqnarray}
Apparently (\ref{Riccati without delay}) is a Riccati equation without delay, and its solvability will be addressed in the next section. (\ref{linear delayed and time-advanced BSDE}) is a linear DABSDE, and it is a hard work to study its solvability, so let's put this question aside for now.

Based on (\ref{delayed Riccati}), (\ref{Riccati without delay}) and (\ref{linear delayed and time-advanced BSDE}), we finally introduce the following ASDDE:
\begin{eqnarray}\left\{\begin{aligned}\label{linear ASDDE}
      dS(t)=&\bigg\{\big[A(t)^\top-L(t)B(t)\Sigma(t)\mathcal{M}^{-1}(t)B(t)^\top-L(t)C(t)\mathcal{N}^{-1}(t)C(t)^\top\big]S(t)\\
      &\quad+\mathbb{E}^{\mathcal{F}_t}\big[(\bar{A}^\top S)|_{t+\delta}\big]-L(t)B(t)\Sigma(t)\mathcal{M}^{-1}(t)\mathbb{E}^{\mathcal{F}_t}\big[(\bar{B}^\top S)|_{t+\delta}\big]\\
      &\quad+L(t)B(t)\big\{2\Sigma(t)[R(t)+\bar{R}(t+\delta)]+I\big\}^{-1}\Gamma(t)\\
      &\quad-L(t)\bar{B}(t)\Sigma(t-\delta)\mathcal{M}^{-1}(t-\delta)(B^\top S)|_{t-\delta}-L(t)\bar{B}(t)\Sigma(t-\delta)\mathcal{M}^{-1}(t-\delta)\\
      &\quad\times\bar{B}(t)^\top S(t)+L(t)\bar{B}(t)\big\{2\Sigma(t-\delta)[R(t-\delta)+\bar{R}(t)]+I\big\}^{-1}\Gamma(t-\delta)\\
      &\quad-L(t)C(t)\mathcal{N}^{-1}(t)\mathbb{E}^{\mathcal{F}_{t}}\big[(\bar{C}^\top S)|_{t+\delta}\big]-L(t)\bar{C}(t)(\mathcal{N}^{-1}C^\top S)|_{t-\delta}\\
      &\quad-L(t)\bar{C}(t)\mathcal{N}^{-1}(t-\delta)\bar{C}(t)^\top S(t)+\big[L(t)B(t)\Sigma(t)\mathcal{M}^{-1}(t)\bar{B}(t+\delta)^\top\\
      &\quad+L(t)C(t)\mathcal{N}^{-1}(t)\bar{C}(t+\delta)^\top-\bar{A}(t+\delta)^\top\big]L(t+\delta)\big[I+(\Sigma L)|_{t+\delta}\big]^{-1}\\
      &\quad\times\mathbb{E}^{\mathcal{F}_t}\big[(\Sigma S+\Lambda)|_{t+\delta}\big]+L(t)
       \big[\bar{B}(t)\Sigma(t-\delta)\mathcal{M}^{-1}(t-\delta)(B^\top L)|_{t-\delta}-\bar{A}(t)\\
      &\quad+\bar{C}(t)(\mathcal{N}^{-1}C^\top L)|_{t-\delta}\big]\big[I+(\Sigma L)|_{t-\delta}\big]^{-1}(\Sigma S+\Lambda)|_{t-\delta}-L(t)\big\{\bar{B}(t)\Sigma(t-\delta)\\
      &\quad\times\mathcal{M}^{-1}(t-\delta)\bar{B}(t)^\top +\bar{C}(t)\mathcal{N}^{-1}(t-\delta)\bar{C}(t)^\top\big\}
       \big[\mathbb{E}^{\mathcal{F}_{t-\delta}}(\bar{X}(t))-\bar{X}(t)\big]\bigg\}dt\\
      &+\bigg\{\big[I+L(t)\Sigma(t)\big]\mathcal{M}^{-1}(t)\big\{B(t)^\top S(t)+\mathbb{E}^{\mathcal{F}_t}\big[(\bar{B}^\top S)|_{t+\delta}\big]\big\}
       -\big[L(t)-2R(t)\\
      &\quad-2\bar{R}(t+\delta)\big]\big\{2\Sigma(t)[R(t)+\bar{R}(t+\delta)]+I\big\}^{-1}\Gamma(t)-\big[I+L(t)\Sigma(t)\big]\mathcal{M}^{-1}(t)\\
      &\quad\times(\bar{B}^\top L)|_{t+\delta}\big[I+(\Sigma L)|_{t+\delta}\big]^{-1}
       \mathbb{E}^{\mathcal{F}_t}\big[(\Sigma S+\Lambda)|_{t+\delta}\big]-\big[I+L(t)\Sigma(t)\big]\mathcal{M}^{-1}(t)\\
      &\quad\times B(t)^\top L(t)\big[I+\Sigma(t)L(t)\big]^{-1}\big[\Sigma(t)S(t)+\Lambda(t)\big]\bigg\}dW(t)\quad t\in[s,T],\\
 S(s)=&\ 0,\ S(t)=0,\quad t\in[s-\delta,s)\cup(T,T+\delta],
\end{aligned}\right.\end{eqnarray}
where $\mathbb{E}^{\mathcal{F}_t}\big[(\bar{A}^\top S)|_{t+\delta}\big]=\mathbb{E}^{\mathcal{F}_t}\big[\bar{A}(t+\delta)^\top S(t+\delta)\big]$ and $(B^\top S)|_{t-\delta}=B(t-\delta)^\top S(t-\delta)$, etc., for simplicity. By Theorem \ref{thm2.2}, the ASDDE (\ref{linear ASDDE}) admits a unique solution $S(\cdot)\in L^2_{\mathcal{F}}(\Omega;C([s,T];\mathbf{R}^n))$.

Next we give the main results of this paper.

\begin{mythm}\label{thm3.1}
Let \textbf{(A1)-(A3)} hold. Suppose that $\bar{A}(t)=\bar{B}(t)=\bar{C}(t)=0$ for $t\in[s,s+\delta]$, $\bar{Q}(t)=\bar{R}(t)=\bar{N}(t)=\bar{A}(t)=\bar{B}(t)=\bar{C}(t)=0$ for $t\in[T,T+\delta]$ and
\begin{equation}\label{important equality}
  \bar{A}(t+\delta)+\bar{B}(t+\delta)\mathcal{R}^{-1}_i(t)B(t)^\top P_i(t)+\bar{C}(t+\delta)\mathcal{N}^{-1}(t)C(t)^\top P_i(t)=0,\ t\in[s-\delta,T],
\end{equation}
where $\mathcal{R}_i(t):=2R(t)+2\bar{R}(t+\delta)+P_i(t)$ and $P_i(\cdot)$ is the solution to the delayed Riccati equation
\begin{equation}\left\{\begin{aligned}\label{P-i}
\dot{P}_i(t)=&\ P_i(t)A(t)+A(t)^\top P_i(t)-2[Q(t)+\bar{Q}(t+\delta)]+P_i(t)B(t)\mathcal{R}_i^{-1}(t)B(t)^\top P_i(t)\\
             &\ +P_i(t)C(t)\mathcal{N}^{-1}(t)C(t)^\top P_i(t)+P_i(t)\bar{B}(t)\mathcal{R}_i^{-1}(t-\delta)\bar{B}(t)^\top P_i(t)\\
             &\ +P_i(t)\bar{C}(t)\mathcal{N}^{-1}(t-\delta)\bar{C}(t)^\top P_i(t),\quad t\in[s,T],\\
      P_i(T)=&\ 2iI,\ P_i(t)=I,\quad t\in[s-\delta,s),
\end{aligned}\right.\end{equation}
for all $i\in\mathbf{N}^+$. Suppose $\Sigma(\cdot)\in C([s,T];\textbf{S}_+^n)$, $L(\cdot)\in C([s,T];\textbf{S}_+^n)$, $(\bar{X}(\cdot),\Lambda(\cdot),\Gamma(\cdot))\in L^2_{\mathcal{F}}(\Omega;\\C([s,T];\textbf{R}^n))\times L^2_{\mathcal{F}}(\Omega;C([s,T];\textbf{R}^n))\times L_{\mathcal{F}}^2([s,T];\textbf{R}^n)$ are the solutions to (\ref{delayed Riccati}), (\ref{Riccati without delay}), (\ref{linear delayed and time-advanced BSDE}), respectively.
Then \textbf{Problem (D-BSLQ)} is uniquely solvable and the optimal control can be expressed as
\begin{equation}\label{optimal control-1}
  u^*(t)=\mathcal{N}^{-1}(t)\Big\{-C(t)^\top L(t)Y^*(t)+C(t)^\top S(t)+\mathbb{E}^{\mathcal{F}_t}\big[(-\bar{C}^\top LY^*+\bar{C}^\top S)|_{t+\delta}\big]\Big\},
\end{equation}
for $t\in[s,T]$. Moreover, the optimal cost is
\begin{equation}\label{optimal cost}\begin{aligned}
  &J(s,\xi;u^*(\cdot))=\mathbb{E}\bigg\{\big\langle\bar{G}\varphi(s-\delta),\varphi(s-\delta)\big\rangle+\big\langle\Lambda(s),(I+2\Sigma(s)G)^{-1}G\Lambda(s)\big\rangle\\
  &\quad+\int_{s-\delta}^s\Big[\big\langle\bar{Q}(t+\delta)\varphi(t),\varphi(t)\big\rangle+\big\langle\bar{R}(t+\delta)\psi(t),\psi(t)\big\rangle
   +\big\langle\bar{N}(t+\delta)\eta(t),\eta(t)\big\rangle\Big]dt\\
  &\quad+\int_s^T\Big[\big\langle[R(t)+\bar{R}(t+\delta)]\big\{2\Sigma(t)[R(t)+\bar{R}(t+\delta)]+I\big\}^{-1}\Gamma(t),\Gamma(t)\big\rangle\\
  &\qquad\qquad+\big\langle[Q(t)+\bar{Q}(t+\delta)]\Lambda(t),\Lambda(t)\big\rangle\Big]dt\bigg\}.
\end{aligned}\end{equation}
\end{mythm}

\begin{Remark}\label{re3.1}
Noting when $\bar{A}(\cdot),\bar{B}(\cdot),\bar{C}(\cdot),\bar{G}(\cdot),\bar{Q}(\cdot),\bar{R}(\cdot),\bar{N}(\cdot)\equiv 0$, Theorem \ref{thm3.1} is reduced to Theorem 3.2 in Lim and Zhou \cite{LZ01}, for the problem without delay. In this case, the optimal control is a linear state feedback of the entire past history of the state process $(Y(\cdot),Z(\cdot))$.
\end{Remark}

\begin{Remark}\label{re3.2}
From (\ref{optimal control-1}), the optimal control $u^*(t)$ explicitly depends on not only the current sate $Y^*(t)$ but also the future state $Y^*(t+\delta)$. Hence it is a linear feedback of the entire past history and the future state trajectory in a short future period of time $\delta$, which is one of the major distinctive features of \textbf{Problem (D-BSLQ)}.
\end{Remark}

\begin{Remark}
The equality (\ref{important equality}) plays an important role in looking for the lower bound of the cost (see (\ref{look for lower bound-1}) in Section 5) and deriving the existence and uniqueness of the solution to the following stochastic Hamiltonian system (see Proposition \ref{pro5.1} in Section 5). Although it seems a little complex, it is not harsh. For example, considering the one-dimensional case, let $R(\cdot),\bar{R}(\cdot),C(\cdot)\equiv 0$ and $\bar{A}(\cdot+\delta)+\bar{B}(\cdot+\delta)B(\cdot)=0$, then (\ref{important equality}) holds.
\end{Remark}

\vspace{1mm}

In fact, the following result can explain the above remark more clearly. For this target, we introduce the following stochastic Hamiltonian system:
\begin{eqnarray}\left\{\begin{aligned}\label{stochastic Hamiltonian}
  dX^*(t)=&\Big\{A(t)^\top X^*(t)-2[Q(t)+\bar{Q}(t+\delta)]Y^*(t)+\mathbb{E}^{\mathcal{F}_t}\big[(\bar{A}^\top X^*)|_{t+\delta}\big]\Big\}dt\\
          &+\Big\{B(t)^\top X^*(t)-2[R(t)+\bar{R}(t+\delta)]Z^*(t)+\mathbb{E}^{\mathcal{F}_t}\big[(\bar{B}^\top X^*)|_{t+\delta}\big]\Big\}dW(t),\\
 -dY^*(t)=&\big[A(t)Y^*(t)+\bar{A}(t)Y^*(t-\delta)+B(t)Z^*(t)+\bar{B}(t)Z^*(t-\delta)+C(t)u^*(t)\\
          &+\bar{C}(t)u^*(t-\delta)\big]dt-Z^*(t)dW(t),\quad t\in[s,T],\\
  X^*(s)=&-2GY^*(s),\ X^*(t)=0,\quad t\in(T,T+\delta],\\
  Y^*(T)=&\ \xi,\ Y^*(t)=\varphi(t),\ Z^*(t)=\psi(t),u^*(t)=\eta(t),\quad t\in[s-\delta,s).
\end{aligned}\right.\end{eqnarray}
Then we give the following result.

\begin{mythm}\label{thm3.2}
Under the same conditions as Theorem \ref{thm3.1}, the stochastic Hamiltonian system (\ref{stochastic Hamiltonian}) has a unique solution $(X^*(\cdot),Y^*(\cdot),Z^*(\cdot))\in L^2_\mathcal{F}(\Omega;C([s,T];\mathbf{R}^n))\times L^2_\mathcal{F}(\Omega;C([s,T];\mathbf{R}^n))\times L^2_\mathcal{F}([s,T];\mathbf{R}^n))$. Moreover, \textbf{Problem (D-BSLQ)} is uniquely solvable and the optimal control can be expressed as
\begin{equation}\label{optimal control-2}
  u^*(t)=\mathcal{N}^{-1}(t)\Big\{C(t)^\top X^*(t)+\mathbb{E}^{\mathcal{F}_t}\big[(\bar{C}^\top X^*)|_{t+\delta})\big]\Big\},\quad t\in[s,T],
\end{equation}
and the optimal cost is (\ref{optimal cost}).
\end{mythm}

The proofs of Theorem \ref{thm3.1} and Theorem \ref{thm3.2} are deferred to Section 5.

\begin{Remark}\label{re3.3}
In the stochastic Hamiltonian system (\ref{stochastic Hamiltonian}), $(Y^*(\cdot),Z^*(\cdot))$ is indeed the state process pair, while $X^*(\cdot)$ is the adjoint process.
\end{Remark}

\section{Unique Solvability of Riccati Equations}

In this section, we will study the existence and uniqueness of solutions to (\ref{delayed Riccati}) and (\ref{Riccati without delay}). It is clear that when $\bar{B}(\cdot)\equiv 0$, (\ref{delayed Riccati}) becomes the same type of Riccati equation as (3.4) in \cite{LZ01} and is uniquely solvable. Hence in this section, we only discuss the case where $\bar{B}(\cdot)\neq 0$.
First we consider the following two equations:
\begin{equation}\left\{\begin{aligned}\label{Sigma}
\dot{\Sigma}(t)=&-\Sigma(t)A(t)^\top-A(t)\Sigma(t)+2\Sigma(t)\big[Q(t)+\bar{Q}(t+\delta)\big]\Sigma(t)\\
                &-B(t)\Sigma(t)\mathcal{M}^{-1}(t)B(t)^\top-C(t)\mathcal{N}^{-1}(t)C(t)^\top\\
                &-\bar{B}(t)\Sigma(t-\delta)\mathcal{M}^{-1}(t-\delta)\bar{B}(t)^\top-\bar{C}(t)\mathcal{N}^{-1}(t-\delta)\bar{C}(t)^\top,\quad t\in[s,T],\\
      \Sigma(T)=&\ M,\ \Sigma(t)=I,\quad t\in[s-\delta,s),
\end{aligned}\right.\end{equation}
\begin{equation}\left\{\begin{aligned}\label{P}
\dot{P}(t)=&\ P(t)A(t)+A(t)^\top P(t)-2\big[Q(t)+\bar{Q}(t+\delta)\big]\\
           &+P(t)B(t)\big[2R(t)+2\bar{R}(t+\delta)+P(t)\big]^{-1}B(t)^\top P(t)\\
           &+P(t)C(t)\mathcal{N}^{-1}(t)C(t)^\top P(t)\\
           &+P(t)\bar{B}(t)\big[2R(t-\delta)+2\bar{R}(t)+P(t-\delta)\big]^{-1}\bar{B}(t)^\top P(t)\\
           &+P(t)\bar{C}(t)\mathcal{N}^{-1}(t-\delta)\bar{C}(t)^\top P(t),\quad t\in[s,T],\\
      P(T)=&\ M^{-1},\ P(t)=I,\quad t\in[s-\delta,s),
\end{aligned}\right.\end{equation}
where $\mathcal{M}(t):=2R(t)\Sigma(t)+2\bar{R}(t+\delta)\Sigma(t)+I$, $\mathcal{N}(t):=2N(t)+2\bar{N}(t+\delta)$ and $M$ is a given $n\times n$ symmetric matrix. In the following context, we will suppress some time variables $t$ for simplicity of writing without ambiguity.
\begin{mypro}\label{pro4.1}
If the solution $\Sigma(\cdot)$ to the Riccati equation (\ref{Sigma}) satisfies $\Sigma(\cdot)\in C([s,T];\mathbf{S}^n_{+})$, then the solution is unique.
\end{mypro}
\begin{proof}
Suppose $\Sigma_1(\cdot),\Sigma_2(\cdot)\in C([s,T];\mathbf{S}_{+}^n)$ are two solutions to (\ref{Sigma}). Denote $\Delta(\cdot):=\Sigma_1(\cdot)-\Sigma_2(\cdot)$, then we have
\begin{equation*}\left\{\begin{aligned}
\dot{\Delta}=&\ \big[-A+2\Sigma_1(Q+\bar{Q}|_{t+\delta})\big]\Delta+\Delta\big[-A+2\Sigma_1(Q+\bar{Q}|_{t+\delta})\big]^\top\\
             &-2\Delta(Q+\bar{Q}|_{t+\delta})\Delta-B\Delta\big[2(R+\bar{R}|_{t+\delta})\Sigma_1+I\big]^{-1}B^\top\\
             &+B\Sigma_2\big[2(R+\bar{R}|_{t+\delta})\Sigma_2+I\big]^{-1}(2R+2\bar{R}|_{t+\delta})\Delta\big[2(R+\bar{R}|_{t+\delta})\Sigma_1+I\big]^{-1}B^\top\\
             &-\bar{B}\Delta|_{t-\delta}\big[2(R|_{t-\delta}+\bar{R})\Sigma_1|_{t-\delta}+I\big]^{-1}\bar{B}^\top
              +\bar{B}\Sigma_2|_{t-\delta}\big[2(R|_{t-\delta}+\bar{R})\Sigma_2|_{t-\delta}+I\big]^{-1}\\
             &\times(2R|_{t-\delta}+2\bar{R})\Delta|_{t-\delta}\big[2(R|_{t-\delta}+\bar{R})\Sigma_1|_{t-\delta}+I\big]^{-1}\bar{B}^\top,\quad t\in[s,T],\\
   \Delta(T)=&\ 0,\ \Delta(t)=0,\quad t\in[s-\delta,s).
\end{aligned}\right.\end{equation*}
For any $\beta\in \mathbf{R}$, applying It\^o's formula to $e^{\beta t}|\Delta(t)|^2$, we have
\begin{equation}\begin{aligned}\label{eq4.3}
  0=&\ e^{\beta t}|\Delta(t)|^2+\int_t^Te^{\beta r}\Big\{\beta|\Delta|^2+2\big\langle\Delta,\big[-A+2\Sigma_1(Q+\bar{Q}|_{r+\delta})\big]\Delta\\
    &\ +\Delta\big[-A+2\Sigma_1(Q+\bar{Q}|_{r+\delta})\big]^\top-2\Delta(Q+\bar{Q}|_{r+\delta})\Delta-B\Delta\big[2(R+\bar{R}|_{r+\delta})\Sigma_1+I\big]^{-1}B^\top\\
    &\ +B\Sigma_2\big[2(R+\bar{R}|_{r+\delta})\Sigma_2+I\big]^{-1}(2R+2\bar{R}|_{r+\delta})\Delta\big[2(R+\bar{R}|_{r+\delta})\Sigma_1+I\big]^{-1}B^\top\\
    &\ -\bar{B}\Delta|_{r-\delta}\big[2(R|_{r-\delta}+\bar{R})\Sigma_1|_{r-\delta}+I\big]^{-1}\bar{B}^\top
     +\bar{B}\Sigma_2|_{r-\delta}[2(R|_{r-\delta}+\bar{R})\Sigma_2|_{r-\delta}+I]^{-1}\\
    &\ \times(2R|_{r-\delta}+2\bar{R})\Delta|_{r-\delta}\big[2(R|_{r-\delta}+\bar{R})\Sigma_1|_{r-\delta}+I\big]^{-1}\bar{B}^\top\big\rangle\Big\}dr.
\end{aligned}\end{equation}
Hence we obtain
\begin{equation*}\begin{aligned}
       &\beta\int_t^Te^{\beta r}|\Delta(r)|^2dr\leq 2\int_t^Te^{\beta r}|\Delta(r)||\dot{\Delta}(r)|dr\\
  \leq &\sup\limits_{s\leq r\leq T}\Big\{4|A|+8|\Sigma_1|(Q+\bar{Q}|_{r+\delta})+4|\Delta|(Q+\bar{Q}|_{r+\delta})+2|B|^2\\
       &\qquad+4|B|^2|\Sigma_2|(R+\bar{R}|_{r+\delta})\Big\}\int_t^Te^{\beta r}|\Delta(r)|^2dr\\
       &+\sup\limits_{s\leq r\leq T}\Big\{|\bar{B}|^2+2|\bar{B}|^2\Sigma_2|_{r-\delta}(R|_{r-\delta}+\bar{R})\Big\}
        \Big\{\int_t^Te^{\beta r}|\Delta(r)|^2dr+\int_t^Te^{\beta r}|\Delta(r-\delta)|^2dr\Big\}.
\end{aligned}\end{equation*}
Noting
\begin{equation*}
  \int_s^Te^{\beta r}|\Delta(r-\delta)|^2dr=e^{\beta\delta}\int_{s-\delta}^{T-\delta}e^{\beta r}|\Delta(r)|^2dr\leq e^{\beta\delta}\int_{s}^{T}e^{\beta r}|\Delta(r)|^2dr.
\end{equation*}
Thus we derive
\begin{equation*}\begin{aligned}
     &\beta\int_s^Te^{\beta r}|\Delta(r)|^2dr\\
\leq &\ \bigg[\sup\limits_{s\leq r\leq T}\Big\{4|A|+8|\Sigma_1|(Q+\bar{Q}|_{r+\delta})+4|\Delta|(Q+\bar{Q}|_{r+\delta})+2|B|^2+4|B|^2|\Sigma_2|(R+\bar{R}|_{r+\delta})\Big\}\\
     &+\sup\limits_{s\leq r\leq T}\Big\{|\bar{B}|^2+2|\bar{B}|^2\Sigma_2|_{r-\delta}(R|_{r-\delta}+\bar{R})\Big\}(1+e^{\beta\delta})\bigg]\int_s^Te^{\beta r}|\Delta(r)|^2dr.
\end{aligned}\end{equation*}
Since $\Sigma_1(\cdot),\Sigma_2(\cdot)$ are continuous on $[s,T]$, $\Sigma_1(\cdot),\Sigma_2(\cdot),\Delta(\cdot)$ are uniformly bounded. Recall \textbf{(A1), (A3)} and $\delta$ is sufficiently small, we can choose $\beta$ sufficiently large such that\begin{equation*}\begin{aligned}
  &\beta>\sup\limits_{s\leq r\leq T}\Big\{4|A|+8|\Sigma_1|(Q+\bar{Q}|_{r+\delta})+4|\Delta|(Q+\bar{Q}|_{r+\delta})+2|B|^2+4|B|^2|\Sigma_2|(R+\bar{R}|_{r+\delta})\Big\}\\ &\qquad+\sup\limits_{s\leq r\leq T}\Big\{|\bar{B}|^2+2|\bar{B}|^2\Sigma_2|_{r-\delta}(R|_{r-\delta}+\bar{R})\Big\}(1+e^{\beta\delta}),
\end{aligned}\end{equation*}
then we deduce
\begin{equation*}
  \int_s^Te^{\beta r}|\Delta(r)|^2dr=0.
\end{equation*}
And it follows that
\begin{equation}\begin{aligned}\label{eq4.4}
  &\Delta(t)=0,\quad a.e.\ t\in[s,T].
\end{aligned}\end{equation}
Finally, by (\ref{eq4.3}) and (\ref{eq4.4}), we get
\begin{equation*}
  e^{\beta t}|\Delta(t)|^2= 0,
\end{equation*}
thus $\Delta(t)=0$ for all $s\leq t\leq T$. The proof is complete.
\end{proof}

If we can prove the existence of solutions to the Riccati equation (\ref{Sigma}), then we can obtain the unique solvability of it. However, this is an arduous work and we can only deal with the special case so far.  Let $R(\cdot)+\bar{R}(\cdot+\delta)\equiv 0$, now (\ref{Sigma}) and (\ref{P}) become:
\begin{equation}\left\{\begin{aligned}\label{Sigma-special}
  \dot{\Sigma}=&-\Sigma A^\top-A\Sigma+2\Sigma(Q+\bar{Q}|_{t+\delta})\Sigma-B\Sigma B^\top-C\mathcal{N}^{-1}C^\top\\
               &-\bar{B}\Sigma|_{t-\delta}\bar{B}^\top-\bar{C}\mathcal{N}^{-1}|_{t-\delta}\bar{C}^\top,\quad t\in[s,T],\\
     \Sigma(T)=&\ M,\ \Sigma(t)=I,\quad t\in[s-\delta,s),
\end{aligned}\right.\end{equation}
\begin{equation}\left\{\begin{aligned}\label{P-special}
  \dot{P}=&\ PA+A^\top P-2(Q+\bar{Q}|_{t+\delta})+PBP^{-1}B^\top P+PC\mathcal{N}^{-1}C^\top P\\
          &+P\bar{B}P^{-1}|_{t-\delta}\bar{B}^\top P+P\bar{C}\mathcal{N}^{-1}|_{t-\delta}\bar{C}^\top P,\quad t\in[s,T],\\
     P(T)=&\ M^{-1},\ P(t)=I,\quad t\in[s-\delta,s).
\end{aligned}\right.\end{equation}

\vspace{1mm}

To prove the unique solvability of (\ref{Sigma-special}), we need the following lemma.

Consider the linear delayed matrix-valued differential equation
\begin{equation}\left\{\begin{aligned}\label{Sigma-special-special}
   \dot{\hat{\Sigma}}(t)=&-\hat{\Sigma}(t)\hat{A}(t)^\top-\hat{A}(t)\hat{\Sigma}(t)-\hat{\Sigma}(t)-\hat{B}(t)\hat{\Sigma}(t-\delta)\hat{B}(t)^\top-\hat{H}(t),\quad t\in[s,T],\\
         \hat{\Sigma}(T)=&\ \hat{M},\ \hat{\Sigma}(t)=\hat{F}(t),\quad t\in[s-\delta,s),
\end{aligned}\right.\end{equation}
where $\hat{A}(\cdot),\hat{B}(\cdot)\in L^\infty([0,T];\mathbf{R}^{n\times n}),\hat{H}(\cdot)\in L^\infty([0,T];\mathbf{S}^n),\hat{F}(\cdot)\in L^\infty([s-\delta,s];\mathbf{S}^n)$ and $\hat{M}$ is a given $n\times n$ symmetric matrix. We introduce the following ASDE:
\begin{equation}\left\{\begin{aligned}\label{linear ASDE}
  d\Phi(t)=&\big\{\hat{A}(t)^\top\Phi(t)+\mathbb{E}^{\mathcal{F}_t}\big[(\hat{B}^\top\Phi)|_{t+\delta}\big]\big\}dt\\
           &+\big\{\Phi(t)-\mathbb{E}^{\mathcal{F}_t}\big[(\hat{B}^\top\Phi)|_{t+\delta}\big]\big\}dW(t),\quad t\in[s,T],\\
   \Phi(s)=&\ I,\ \Phi(t)=0,\quad t\in(T,T+\delta].
\end{aligned}\right.\end{equation}
By Theorem \ref{thm2.2}, the ASDE (\ref{linear ASDE}) has a unique solution $\Phi(\cdot)\in L^2_{\mathcal{F}}(\Omega;C([s,T];\mathbf{R}^{n\times n}))$.
In the following, we aim to obtain an explicit solution to (\ref{linear ASDE}). To this end, let $\Pi(\cdot)$ be an exponential martingale, which satisfies
\begin{equation*}\left\{\begin{aligned}
  d\Pi(t)&=\gamma(t)\Pi(t)dW(t),\quad t\in[s,T+\delta],\\
   \Pi(s)&=I,
\end{aligned}\right.\end{equation*}
and $\Phi(t)=\Upsilon(t)\Pi(t)$, where $\gamma(\cdot)$ and $\Upsilon(\cdot)$ are deterministic functions to be determined. Applying It\^o's formula to $\Phi(\cdot)$, we obtain
\begin{equation}\left\{\begin{aligned}\label{Upsilon}
  &\dot{\Upsilon}(t)=\hat{A}(t)^\top\Upsilon(t)+\hat{B}(t+\delta)^\top\Upsilon(t+\delta),\\
  &\Upsilon(t)\gamma(t)=\Upsilon(t)-\hat{B}(t+\delta)^\top\Upsilon(t+\delta),\quad t\in[s,T].
\end{aligned}\right.\end{equation}
Hence we derive the following {\it time-advanced ordinary differential equation} (AODE, for short):
\begin{equation}\left\{\begin{aligned}\label{linear AODE}
  &\dot{\Upsilon}(t)=\hat{A}(t)^\top\Upsilon(t)+\hat{B}(t+\delta)^\top\Upsilon(t+\delta),\quad t\in[s,T],\\
  &\Upsilon(s)=I,\ \Upsilon(t)=0,\quad t\in(T,T+\delta].
\end{aligned}\right.\end{equation}
Since the coefficients of (\ref{linear AODE}) are bounded, it admits a unique solution $\Upsilon(\cdot)\in C([s,T];\mathbf{R}^{n\times n})$. Then by (\ref{Upsilon}), we can get $\gamma(\cdot)$. Furthermore, if $\Upsilon(\cdot)>0$, then $\Phi(\cdot)>0$.

\begin{mylem}\label{lem4.1}
The linear delayed matrix-valued differential equation (\ref{Sigma-special-special}) admits a unique solution $\hat{\Sigma}(\cdot)\in C([s,T];\mathbf{S}^n)$. Furthermore, if $\hat{H}(\cdot)\geq 0$, $\hat{F}(\cdot)\geq 0$, $\hat{M}\geq 0$ and $\hat{A}(\cdot),\hat{B}(\cdot)$ satisfy that the solution to the AODE (\ref{linear AODE}) $\Upsilon(\cdot)>0$, then $\hat{\Sigma}(\cdot)\in C([s,T];\mathbf{S}^n_+)$.
\end{mylem}
\begin{proof}
Since the coefficients of (\ref{Sigma-special-special}) are bounded, the unique solvability of (\ref{Sigma-special-special}) can be proved in the same method as Theorem \ref{thm2.1}. Next we prove the second result. Applying It\^o's formula to $\Phi(\cdot)^\top\hat{\Sigma}(\cdot)\Phi(\cdot)$, we have
\begin{equation*}\begin{aligned}
\Phi(t)^\top\hat{\Sigma}(t)\Phi(t)=&\ \Phi(T)^\top\hat{M}\Phi(T)+\int_t^T\Big\{\Phi(r)^\top\big[\hat{B}(r)\hat{\Sigma}(r-\delta)\hat{B}(r)^\top+\hat{H}(r)\big]\Phi(r)\\
 &-\mathbb{E}^{\mathcal{F}_t}\big[(\Phi^\top\hat{B})|_{r+\delta}\big]\hat{\Sigma}(r)\mathbb{E}^{\mathcal{F}_t}\big[(\hat{B}^\top\Phi)|_{r+\delta}\big]\Big\}dr
 -\int_t^T\Big\{\cdots\Big\}dW(r).
\end{aligned}\end{equation*}
Noting $\Phi(t)=0$ for $t\in(T,T+\delta]$, we get
\begin{equation*}\begin{aligned}
  &\mathbb{E}\int_t^T(\Phi^{-1}(t))^\top\Phi(r)^\top\hat{B}(r)\hat{\Sigma}(r-\delta)\hat{B}(r)\Phi(r)\Phi^{-1}(t)dr\\
 =&\ \mathbb{E}\int_{t-\delta}^t(\Phi^{-1}(t))^\top(\Phi^\top\hat{B})|_{r+\delta}\hat{\Sigma}(r)(\hat{B}^\top\Phi)|_{r+\delta}\Phi^{-1}(t)dr\\
  &+\mathbb{E}\int_t^T(\Phi^{-1}(t))^\top(\Phi^\top\hat{B})|_{r+\delta}\hat{\Sigma}(r)(\hat{B}^\top\Phi)|_{r+\delta}\Phi^{-1}(t)dr.
\end{aligned}\end{equation*}
Since $\Upsilon(\cdot)>0$, $\Phi(\cdot)>0$, we derive
\begin{equation*}\begin{aligned}
 \hat{\Sigma}(t)=&\ \mathbb{E}\bigg[(\Phi^{-1}(t))^\top\Phi(T)^\top\hat{M}\Phi(T)\Phi^{-1}(t)+\int_t^T(\Phi^{-1}(t))^\top\Phi(r)^\top\hat{H}(r)\Phi(r)\Phi^{-1}(t)dr\\
                 &\quad+\int_{t-\delta}^t(\Phi^{-1}(t))^\top(\Phi^\top\hat{B})|_{r+\delta}\hat{\Sigma}(r)(\hat{B}^T\Phi)|_{r+\delta}\Phi^{-1}(t)dr\bigg].
\end{aligned}\end{equation*}
Consequently $\hat{\Sigma}(s)\geq 0$ if $\hat{H}(\cdot)\geq 0$, $\hat{M}\geq 0$ and $\hat{F}(\cdot)\geq 0$.
Similarly, for fixed $t\in(s,T]$, $\hat{\Sigma}(t)\geq 0$ because $\hat{\Sigma}(r)\geq 0$ for $r\in[t-\delta,t)$. The proof is complete.
\end{proof}

\begin{Remark}\label{re4.1}
In fact, from the proof of Lemma \ref{lem4.1}, it follows that if $\hat{M}>0$, then $\hat{\Sigma}(\cdot)>0$.
\end{Remark}

\begin{Remark}\label{re4.1-1}
If we consider the solvability of the following equation:
\begin{equation*}\left\{\begin{aligned}
   \dot{\hat{\Sigma}}(t)=&-\hat{\Sigma}(t)\hat{A}(t)^\top-\hat{A}(t)\hat{\Sigma}(t)-B(t)\hat{\Sigma}(t)B(t)^T-\hat{B}(t)\hat{\Sigma}(t-\delta)\hat{B}(t)^\top-\hat{H}(t),\quad t\in[s,T],\\
         \hat{\Sigma}(T)=&\ \hat{M},\ \hat{\Sigma}(t)=\hat{F}(t),\quad t\in[s-\delta,s),
\end{aligned}\right.\end{equation*}
then we can't find $\Phi(\cdot)$ satisfying some equation such that $\hat{\Sigma}(\cdot)\geq 0$. Hence we only give the solvability in the special case: $B(\cdot)=I$ (as Lemma \ref{lem4.1}).
\end{Remark}

Next we continue to study the existence of solutions to (\ref{Sigma-special}). Let {\color{red}$B(\cdot)=I$} and $M>0$. Denote
\begin{equation}\left\{\begin{aligned}\label{A}
  &\hat{A}(t):=A(t)-2\Sigma(t)\big[Q(t)+\bar{Q}(t+\delta)\big],\\
  &\Psi(t):=2\big[Q(t)+\bar{Q}(t+\delta)\big]^{\frac{1}{2}}\Sigma(t),\\
  &\hat{H}(t):=C(t)\mathcal{N}^{-1}(t)C(t)^\top+\bar{C}(t)\mathcal{N}^{-1}(t-\delta)\bar{C}(t)^\top+2\Sigma(t)\big[Q(t)+\bar{Q}(t+\delta)\big]\Sigma(t).
\end{aligned}\right.\end{equation}
Apparently (\ref{Sigma-special}) is equivalent to the following equation:
\begin{equation}\left\{\begin{aligned}\label{Sigma-special-special-special}
   \dot{\Sigma}=&-\Sigma\hat{A}^\top-\hat{A}\Sigma-B\Sigma B^\top-\bar{B}\Sigma|_{t-\delta}\bar{B}^\top-\hat{H},\quad t\in[s,T],\\
      \Sigma(T)=&\ M,\ \Sigma(t)=I,\quad t\in[s-\delta,s).
\end{aligned}\right.\end{equation}
Next we construct the iterative scheme as follows. For $i=0,1,2,\cdots$, set
\begin{equation}\left\{\begin{aligned}\label{A-i}
  &\Sigma^0(t)=M\ \mbox{for }t\in[s,T],\ \Sigma^0(t)=I\ \mbox{for }t\in[s-\delta,s),\\
  &\hat{A}_i(t)=A(t)-2\Sigma^i(t)\big[Q(t)+\bar{Q}(t+\delta)\big],\\
  &\Psi_i(t)=2\big[Q(t)+\bar{Q}(t+\delta)\big]^{\frac{1}{2}}\Sigma^i(t),\\
  &\hat{H}_i(t)=C(t)\mathcal{N}^{-1}(t)C(t)^\top+\bar{C}(t)\mathcal{N}^{-1}(t-\delta)\bar{C}(t)^\top+2\Sigma^i(t)\big[Q(t)+\bar{Q}(t+\delta)\big]\Sigma^i(t),
\end{aligned}\right.\end{equation}
and $\Sigma^{i+1}(\cdot)$ be the solution to
\begin{equation}\left\{\begin{aligned}\label{Sigma-i}
  \dot{\Sigma}^{i+1}=&-\Sigma^{i+1}\hat{A}_i^\top-\hat{A}_i\Sigma^{i+1}-B\Sigma^{i+1}B^\top-\bar{B}\Sigma^{i+1}|_{t-\delta}\bar{B}^\top-\hat{H}_i,\quad t\in[s,T],\\
     \Sigma^{i+1}(T)=&\ M,\ \Sigma^{i+1}(t)=I,\quad t\in[s-\delta,s).
\end{aligned}\right.\end{equation}
Let $\Phi_i(\cdot)$ and $\Upsilon_i(\cdot)$ be the solutions to ASDE (\ref{linear ASDE}) and AODE (\ref{linear AODE}) associated with $\hat{A}_i(\cdot)$ and $\bar{B}(\cdot)$, respectively. Suppose $\Upsilon_i(\cdot)>0$, then by Lemma \ref{lem4.1} (noting $B(\cdot)\equiv I$), we get $\Sigma^{i}(\cdot)\in C([s,T];\bar{\mathbf{S}}^n_+)$. Denote $\Delta^i(\cdot):=\Sigma^i(\cdot)-\Sigma^{i+1}(\cdot)$, then we have
\begin{equation*}\left\{\begin{aligned}
  -\dot{\Delta}^i=&\ \Delta^i\hat{A}_i^\top+\hat{A}_i\Delta^i+B\Delta^iB^\top+\bar{B}\Delta^i|_{t-\delta}\bar{B}^\top+\Sigma^i(\hat{A}_{i-1}-\hat{A}_i)^\top\\
                  &+(\hat{A}_{i-1}-\hat{A}_i)\Sigma^i+\hat{H}_{i-1}-\hat{H}_i,\quad t\in[s,T],\\
      \Delta^i(T)=&\ 0,\ \Delta^i(t)=0,\quad t\in[s-\delta,s).
\end{aligned}\right.\end{equation*}
Denote $\Theta_i(t):=2\big[Q(t)+\bar{Q}(t+\delta)\big]^{\frac{1}{2}}(\Sigma^i(t)-\Sigma^{i-1}(t))$, noting
\begin{equation*}\left\{\begin{aligned}
  &\hat{H}_{i-1}(t)-\hat{H}_i(t)=\frac{1}{2}\Theta_i(t)^\top\Theta_i(t)-\frac{1}{2}\Psi_i(t)^\top\Theta_i(t)-\frac{1}{2}\Theta_i(t)^\top\Psi_i(t),\\
  &\hat{A}_{i-1}(t)-\hat{A}_i(t)=\Theta_i(t)^\top\big[Q(t)+\bar{Q}(t+\delta)\big]^{\frac{1}{2}}.
\end{aligned}\right.\end{equation*}
Hence we obtain
\begin{equation*}
  -\big[\dot{\Delta}_i+\Delta^i\hat{A}_i^\top+\hat{A}_i\Delta^i+B\Delta^iB^\top+\bar{B}\Delta^i|_{t-\delta}\bar{B}^\top\big]=\frac{1}{2}\Theta_i^\top\Theta_i\geq 0.
\end{equation*}
Since $\Delta^i(T)=0$ and $\Delta^i(t)=0$ for $t\in[s-\delta,s)$, it follows that $\Delta^i(\cdot)\geq 0$. Thus $\{\Sigma^i(\cdot)\}$ is a decreasing sequence in $C([s,T];\bar{\mathbf{S}}^n_+)$. Therefore it has a limit, denoting by $\Sigma(\cdot)$. Clearly $\Sigma(\cdot)\in C([s,T];\mathbf{S}^n_+)$ is the solution to (\ref{Sigma-special-special-special}), hence (\ref{Sigma-special}). Furthermore let $\Phi(\cdot)$ and $\Upsilon(\cdot)$ be the solutions to ASDE (\ref{linear ASDE}) and AODE (\ref{linear AODE}) associated with $\hat{A}(\cdot)$ and $\bar{B}(\cdot)$, respectively. Suppose $\Upsilon(\cdot)>0$, then by Remark \ref{re4.1}, $\Sigma(\cdot)>0$ owing to $M>0$.

If $M\geq 0$, we repeat the above step using $M\geq 0$ instead of $M>0$. Then we get $\Sigma(\cdot)\geq 0$. Finally we summarize the above results as follows:\\
\textbf{(A4)}: For $i=0,1,2,\cdots$, the solution to AODE (\ref{linear AODE}) $\Upsilon_i(\cdot)>0$, where $\hat{A}_i(\cdot)$ is given by (\ref{A-i}) and (\ref{Sigma-i}).\\
\textbf{(A5)}: The solution to AODE (\ref{linear AODE}) $\Upsilon(\cdot)>0$, where $\hat{A}(\cdot)$ is given by (\ref{A}) and (\ref{Sigma-special-special-special}).

\begin{mypro}\label{pro4.2}
Let $M\geq 0$ be $n\times n$ symmetric matrix, suppose \textbf{(A4)} holds and $B(\cdot)=I$, then (\ref{Sigma-special}) has the unique solution $\Sigma(\cdot)\in C([s,T];\mathbf{S}^n_+)$. Furthermore, suppose \textbf{(A5)} holds, then
\begin{description}
  \item[(i)] when $M>0$, $\Sigma(\cdot)\in C([s,T];\bar{\mathbf{S}}^n_+)$;
  \item[(ii)] when $M\geq0$, $\Sigma(\cdot)\in C([s,T];\mathbf{S}^n_+)$.
\end{description}
\end{mypro}

\begin{Remark}\label{re4.2}
The condition \textbf{(A4)}, imposed to the coefficients of (\ref{Sigma-special}), is mainly to guarantee $\Sigma(\cdot)\geq 0$.
\end{Remark}
\begin{Remark}
Although \textbf{(A4)} and \textbf{(A5)} seem a little complex, (\ref{linear AODE}) is a linear ODE, and we have rich results to study its solution. Therefore it is not difficult to study the specific conditions imposed to the coefficients of (\ref{Sigma-special}) to satisfy \textbf{(A4)} and \textbf{(A5)}. However it is not our major job in this paper, so we leave this question to the interested readers.
\end{Remark}

\begin{mypro}\label{pro4.3}
Let $M>0$ be an $n\times n$ symmetric matrix, suppose \textbf{(A4)} and \textbf{(A5)} hold, $B(\cdot)=I$, then (\ref{P-special}) is uniquely solvable with the solution $P(\cdot)\in C([s,T];\bar{\mathbf{S}}^n_+)$.
\end{mypro}
\begin{proof}
By Proposition \ref{pro4.2}, (\ref{Sigma-special}) is uniquely solvable with the solution $\Sigma(\cdot)\in C([s,T];\bar{\mathbf{S}}^n_+)$, hence $P(\cdot)=\Sigma^{-1}(\cdot)$ is well-defined. By calculating, it is indeed a solution to (\ref{P-special}). As for the uniqueness, suppose $P_1(\cdot),P_2(\cdot)$ are two solutions to (\ref{P-special}), then $P_1^{-1}(\cdot),P_2^{-1}(\cdot)$ are two solutions to (\ref{Sigma-special}). Thus by the uniqueness of solutions to (\ref{Sigma-special}), we have $P_1(\cdot)=P_2(\cdot)$. The proof is complete.
\end{proof}

Proposition \ref{pro4.2} implies that (\ref{delayed Riccati}) admits the unique solution $\Sigma(\cdot)\in C([s,T];\mathbf{S}^n_{+})$ when $R(\cdot)+\bar{R}(\cdot+\delta)\equiv 0$. Next we focus on the Riccati equation (\ref{Riccati without delay}).

\begin{mypro}\label{pro4.6}
Let $\Sigma(\cdot)$ be the solution to (\ref{delayed Riccati}), then Riccati equation (\ref{Riccati without delay}) is uniquely solvable, and
\begin{description}
  \item[(i)] if $G>0$, then $L(\cdot)\in C([s,T];\bar{\mathbf{S}}^n_+)$;
  \item[(ii)] if $G\geq 0$, then $L(\cdot)\in C([s,T];\mathbf{S}^n_+)$.
\end{description}
\end{mypro}

\begin{proof}
By making the time reversing transformation $\tau=T-t+s$, we have
\begin{eqnarray}\left\{\begin{aligned}\label{L-special}
  \dot{L}=&-A^\top L-LA-2(Q+\bar{Q}|_{t+\delta})+L\big\{B\Sigma\mathcal{M}^{-1}B^\top+C\mathcal{N}^{-1}C^\top\\
          &+\bar{C}\mathcal{N}^{-1}|_{t-\delta}\bar{C}^\top+\bar{B}(\Sigma\mathcal{M}^{-1})|_{t-\delta}
           \bar{B}^\top\big\}L,\quad t\in[s,T],\\
     L(T)=&\ 2G.
\end{aligned}\right.\end{eqnarray}
When $M_i=\frac{1}{i}I,i\in\mathbf{N}^+$, let $\Sigma_i(\cdot)$ and $P_i(\cdot)$ denote the solutions to (\ref{delayed Riccati}) and (\ref{Riccati without delay}), respectively. Noting
\begin{equation*}\begin{aligned}
  &\Sigma(t)\mathcal{M}^{-1}(t)=\Sigma(t)\big[2(R(t)+\bar{R}(t+\delta))\Sigma(t)+I\big]^{-1}\\
  &=\lim\limits_{i\rightarrow\infty}\Sigma_i(t)\big[2(R(t)+\bar{R}(t+\delta))\Sigma_i(t)+I\big]^{-1}=\lim\limits_{i\rightarrow\infty}\big[2(R(t)+\bar{R}(t+\delta))+P_i(t)\big]^{-1},
\end{aligned}\end{equation*}
thus $B(t)\Sigma(t)\mathcal{M}^{-1}(t)B(t)^\top+C(t)\mathcal{N}^{-1}(t)C(t)^\top+\bar{C}(t)\mathcal{N}^{-1}(t-\delta)\bar{C}(t)^\top
+\bar{B}(t)\Sigma(t-\delta)\mathcal{M}^{-1}(t-\delta)\bar{B}(t)^\top$ is symmetric, for all $t\in[s,T]$. Therefore (\ref{L-special}) is a standard Riccati equation and the proof is completed.
\end{proof}

Next we try to analyze the asymptotic behaviour of the above equations with respect to their terminal conditions, which plays an important role in the proof of Theorem \ref{thm3.1} and \ref{thm3.2}.

For $i\in\mathbf{N}^+$, consider the following delayed Riccati equations
\begin{equation}\left\{\begin{aligned}\label{Sigma-ii}
  \dot{\Sigma}_i=&-\Sigma_i A^\top-A\Sigma_i+2\Sigma_i(Q+\bar{Q}|_{t+\delta})\Sigma_i-B\mathcal{R}_i^{-1}B^\top-C\mathcal{N}^{-1}C^\top\\
                 &-\bar{B}\mathcal{R}_i^{-1}|_{t-\delta}\bar{B}^\top-\bar{C}\mathcal{N}^{-1}|_{t-\delta}\bar{C}^\top,\quad t\in[s,T],\\
     \Sigma_i(T)=&\ \frac{1}{2i}I,\ \Sigma(t)=I,\quad t\in[s-\delta,s),
\end{aligned}\right.\end{equation}
\begin{equation}\left\{\begin{aligned}\label{P-ii}
  \dot{P}_i=&\ P_iA+A^\top P_i-2(Q+\bar{Q}|_{t+\delta})+P_iB\mathcal{R}_i^{-1}B^\top P_i+P_iC\mathcal{N}^{-1}C^\top P_i\\
            &+P_i\bar{B}\mathcal{R}_i^{-1}|_{t-\delta}\bar{B}^\top P_i+P_i\bar{C}\mathcal{N}^{-1}|_{t-\delta}\bar{C}^\top P_i,\quad t\in[s,T],\\
     P_i(T)=&\ 2iI,\ P_i(t)=I,\quad t\in[s-\delta,s),
\end{aligned}\right.\end{equation}
where $\mathcal{R}_i(\cdot):=2R(\cdot)+2\bar{R}(\cdot+\delta)+P_i(\cdot)$. Apparently $\Sigma_i(\cdot)=P^{-1}_i(\cdot)$. By Proposition \ref{pro4.2} and \ref{pro4.3}, suppose \textbf{(A4)} and \textbf{(A5)} hold, $B(\cdot)=I$, $R(\cdot)+\bar{R}(\cdot+\delta)\equiv 0$, then the delayed Riccati equations (\ref{Sigma-ii}) and (\ref{P-ii}) admit unique solutions, respectively. Since our aim is to analyze the asymptotic behaviour of the above equations, we discard the conditions guaranteeing the existence and uniqueness of their solutions in the following text.

\begin{mypro}\label{pro4.7}
For $i\in\mathbf{N}^+$, let $\Sigma_i(\cdot)$ and $\Sigma(\cdot)$ be the solutions to (\ref{Sigma-ii}) and (\ref{delayed Riccati}), respectively, then $\Sigma_i(\cdot)\rightarrow\Sigma(\cdot)$, uniformly on $[s,T]$ as $i\rightarrow \infty$.
\end{mypro}
\begin{proof}
For simplicity of writing, we consider the simple case: $R(t)+\bar{R}(t+\delta)=0$ for $t\in[s-\delta,T]$. For the general case $R(t)+\bar{R}(t+\delta)\geq0$ for $t\in[s-\delta,T]$, the proof is similar.
For $i\in\mathbf{N}^+$, denote $\Delta_i(\cdot):=\Sigma_i(\cdot)-\Sigma(\cdot)$, then we have
\begin{equation*}\left\{\begin{aligned}
  \dot{\Delta}_i=&\ \big[-A+2\Sigma(Q+\bar{Q}|_{t+\delta})\big]\Delta_i+\Delta_i\big[-A+2\Sigma(Q+\bar{Q}|_{t+\delta})\big]^\top\\
                 &+2\Delta_i(Q+\bar{Q}|_{t+\delta})\Delta_i-B\Delta_iB^\top-\bar{B}\Delta_i|_{t-\delta}\bar{B}^\top,\quad t\in[s,T],\\
     \Delta_i(T)=&\ \frac{1}{2i}I,\ \Delta_i(t)=0,\quad t\in[s-\delta,s).
\end{aligned}\right.\end{equation*}
Apparently it is a delayed Riccati equation, and $\Delta_i(\cdot)\in C([s,T];\bar{\mathbf{S}}^n_+)$. Let $\bar{\Sigma}(\cdot)$ be the solution to (\ref{Sigma-ii}) with the terminal condition $\bar{\Sigma}(T)=I$. It is easy to verify that $\Sigma_i(\cdot)<\bar{\Sigma}(\cdot)$ for all $i\in\mathbf{N}^+$. For any $\beta>0$, applying It\^o's formula to $e^{\beta t}|\Delta_i(t)|^2$, we have
\begin{equation}\begin{aligned}\label{eq4.24}
  e^{\beta t}|\Delta_i(t)|^2=e^{\beta T}\frac{1}{4i^2}-\int_t^Te^{\beta r}\big[\beta|\Delta_i(r)|^2+2\langle\Delta_i(r),\dot{\Delta}_i(r)\rangle\big]dr.
\end{aligned}\end{equation}
Thus we derive
\begin{equation*}\begin{aligned}
       &\beta\int_t^Te^{\beta r}|\Delta_i(r)|^2dr\leq \frac{1}{4i^2}e^{\beta T}+\sup\limits_{s\leq r\leq T}\Big\{4|A(r)|+8|\Sigma(r)||Q(r)+\bar{Q}(r+\delta)|\\
       &\quad+4|\Sigma(r)+\bar{\Sigma}(r)||Q(r)+\bar{Q}(r+\delta)|+2|B(r)|^2\Big\}\int_t^Te^{\beta r}|\Delta_i(r)|^2dr\\
       &\quad+(1+e^{\beta\delta})\sup\limits_{s\leq r\leq T}|\bar{B}(r)|^2\int_t^Te^{\beta r}|\Delta_i(r)|^2dr
        +e^{\beta\delta}\sup\limits_{s\leq r\leq T}|\bar{B}(r)|^2\int_{t-\delta}^te^{\beta r}|\Delta_i(r)|^2dr.
\end{aligned}\end{equation*}
We can choose $\beta$ sufficiently large, since $\delta$ is sufficiently small. Thus we get $\int_s^Te^{\beta r}|\Delta_i(r)|^2dr\leq\frac{K}{4i^2}e^{\beta T}$ for some constant $K>0$. Therefore, $\lim\limits_{i\rightarrow\infty}\int_s^Te^{\beta r}|\Delta_i(r)|^2dr=0$, hence $\lim\limits_{i\rightarrow\infty}\Delta_i(r)=0$, a.e. $r\in[s,T]$. Substituting this into (\ref{eq4.24}), we get
$$\lim\limits_{i\rightarrow\infty}\sup\limits_{s\leq t\leq T}e^{\beta t}|\Delta_i(t)|^2=0,$$
thus the result is proved.
\end{proof}

\section{Proofs of Theorem \ref{thm3.1} and Theorem \ref{thm3.2} }

In this section, we proceed to complete the proofs of Theorem \ref{thm3.1} and Theorem \ref{thm3.2}. The basic idea is first to find the lower bound of the cost and next to prove that the control $u^*(\cdot)$ of (\ref{optimal control-1}) achieves exactly its lower bound.

For $i\in\mathbf{N}^+$, recall delayed Riccati equations (\ref{Sigma-ii}), (\ref{P-ii}) and consider the DABSDE:
\begin{equation}\left\{\begin{aligned}\label{linear DABSDE}
  d\bar{X}_i(t)=&\Big\{\big[A^\top-2(Q+\bar{Q}|_{t+\delta})\Sigma_i\big]\bar{X}_i+2(Q+\bar{Q}|_{t+\delta})\Lambda_i
                 +\mathbb{E}^{\mathcal{F}_t}\big[(\bar{A}^\top\bar{X}_i)|_{t+\delta}\big]\Big\}dt\\
                &\ +\Big\{\big[B^\top-2[R+\bar{R}|_{t+\delta}]\mathcal{R}_i^{-1}B^\top\big]\bar{X}_i
                 +\mathbb{E}^{\mathcal{F}_t}\big\{\big[(\bar{B}^\top\bar{X}_i)|_{t+\delta}-2(R+\bar{R}|_{t+\delta})\mathcal{R}_i^{-1}\\
                &\ \times(\bar{B}^\top\bar{X}_i)|_{t+\delta}\big\}+2(R+\bar{R}|_{t+\delta})\big[2\Sigma_i(R+\bar{R}|_{t+\delta})+I\big]^{-1}\Gamma_i\Big\}dW(t),\\
  d\Lambda_i(t)=&\Big\{-A\Lambda_i+2P_i^{-1}(Q+\bar{Q}|_{t+\delta})\Lambda_i-B\mathcal{R}_i^{-1}P_i\Gamma_i
                 +P_i^{-1}\mathbb{E}^{\mathcal{F}_t}\big[(\bar{A}^\top P_i\Lambda_i)|_{t+\delta}\big]\\
                &+B\mathcal{R}_i^{-1}\mathbb{E}^{\mathcal{F}_t}\big[(\bar{B}^\top P_i\Lambda_i)|_{t+\delta}\big]
                 +C\mathcal{N}^{-1}\mathbb{E}^{\mathcal{F}_t}\big[(\bar{C}^\top P_i\Lambda_i)|_{t+\delta}\big]-\bar{B}(\mathcal{R}_i^{-1}P_i\Gamma_i)|_{t-\delta}\\
                &+\bar{B}(\mathcal{R}_i^{-1}B^\top P_i\Lambda_i)|_{t-\delta}
                 +\bar{C}(\mathcal{N}^{-1}C^\top P_i\Lambda_i)|_{t-\delta}+\big[\bar{B}\mathcal{R}_i^{-1}|_{t-\delta}\bar{B}^\top\\
                &+\bar{C}\mathcal{N}^{-1}|_{t-\delta}\bar{C}^\top\big]\big[\mathbb{E}^{\mathcal{F}_{t-\delta}}(\bar{X}_i)-\bar{X}_i\big]\Big\}dt+\Gamma_i(t)dW(t),\quad t\in[s,T],\\
   \bar{X}_i(s)=&\ 2G(I+2\Sigma_i(s)G)^{-1}\Lambda_i(s),\ \bar{X}_i(t)=0,\quad t\in(T,T+\delta],\\
   \Lambda_i(T)=&-\xi,\ \Lambda_i(t)=0,\ \Gamma_i(t)=0,\ t\in[s-\delta,s),\ \Lambda_i(t)=0,\ t\in(T,T+\delta].
\end{aligned}\right.\end{equation}
For more information about DABSDE, refer to He et al. \cite{HRZ20}. Similar to (\ref{linear delayed and time-advanced BSDE}), now we put aside the study of solvability of (\ref{linear DABSDE}) and just assume (\ref{linear DABSDE}) has a unique solution $(\bar{X}_i(\cdot),\Lambda_i(\cdot),\Gamma_i(\cdot))\in L^2_{\mathcal{F}}(\Omega;C([s,T];\mathbf{R}^n))\times L^2_{\mathcal{F}}(\Omega;C([s,T];\mathbf{R}^n))\times L^2_{\mathcal{F}}([s,T];\mathbf{R}^n)$.

Applying It\^o's formula to $\big\langle P_i(\cdot)(Y(\cdot)+\Lambda_i(\cdot)),Y(\cdot)+\Lambda_i(\cdot)\big\rangle$, we have (suppressing $t$)
\begin{equation*}\begin{aligned}
&d\big\langle P_i(Y+\Lambda_i),Y+\Lambda_i\big\rangle=\Big\{\big\langle Y+\Lambda_i,\big[P_iA+A^\top P_i-2(Q+\bar{Q}|_{t+\delta})+P_iB\mathcal{R}_i^{-1}B^\top P_i\\
&\quad+P_iC\mathcal{N}^{-1}C^\top P_i+P_i\bar{B}\mathcal{R}_i^{-1}|_{t-\delta}\bar{B}^\top P_i+P_i\bar{C}\mathcal{N}^{-1}|_{t-\delta}\bar{C}^\top P_i\big](Y+\Lambda_i)\big\rangle\\
&\quad+2\big\langle P_i(Y+\Lambda_i),-AY-\bar{A}Y|_{t-\delta}-BZ-\bar{B}Z|_{t-\delta}-Cu-\bar{C}u|_{t-\delta}-A\Lambda_i\\
&\quad+2P_i^{-1}(Q+\bar{Q}|_{t+\delta})\Lambda_i-B\mathcal{R}_i^{-1}P_i\Gamma_i
 +P_i^{-1}\mathbb{E}^{\mathcal{F}_t}\big[(\bar{A}^\top P_i\Lambda_i)|_{t+\delta}\big]
 +B\mathcal{R}_i^{-1}\mathbb{E}^{\mathcal{F}_t}\big[(\bar{B}^\top P_i\Lambda_i)|_{t+\delta}\big]\\
&\quad+C\mathcal{N}^{-1}\mathbb{E}^{\mathcal{F}_t}\big[(\bar{C}^\top P_i\Lambda_i)|_{t+\delta}\big]-\bar{B}(\mathcal{R}_i^{-1}P_i\Gamma_i)|_{t-\delta}
 +\bar{B}(\mathcal{R}_i^{-1}B^\top P_i\Lambda_i)|_{t-\delta}\\
&\quad+\bar{C}(\mathcal{N}^{-1}C^\top P_i\Lambda_i)|_{t-\delta}+\{\bar{B}\mathcal{R}_i^{-1}|_{t-\delta}\bar{B}^\top+\bar{C}\mathcal{N}^{-1}|_{t-\delta}\bar{C}^\top\}[\mathbb{E}^{\mathcal{F}_{t-\delta}}(\bar{X}_i)-\bar{X}_i]\big\rangle\\
&\quad+\big\langle P_i(Z+\Gamma_i),Z+\Gamma_i\big\rangle\Big\}dt+2\big\langle P_i(Y+\Lambda_i),Z+\Gamma_i\big\rangle dW(t).
\end{aligned}\end{equation*}
Since $\bar{B}(t)=\bar{C}(t)=0$, $t\in[s,s+\delta]$ and $\bar{B}(t)=\bar{C}(t)=\bar{Q}(t)=\bar{R}(t)=\bar{N}(t)=0$, $t\in[T,T+\delta]$, we get
\begin{equation*}\begin{aligned}
 &\mathbb{E}\int_s^T\Big[2\big\langle\bar{Q}(t)Y(t-\delta),Y(t-\delta)\big\rangle+2\big\langle\bar{R}(t)Z(t-\delta),Z(t-\delta)\big\rangle
  +2\big\langle\bar{N}(t)u(t-\delta),u(t-\delta)\big\rangle\\
 &\qquad\quad-2\langle Y(t)+\Lambda_i(t),P_i(t)\bar{B}(t)Z(t-\delta)+P_i(t)\bar{C}(t)u(t-\delta)\rangle\Big]dt\\
=&\ \mathbb{E}\int_{s}^{T}\Big[2\big\langle\bar{Q}(t+\delta)Y(t),Y(t)\big\rangle+2\big\langle\bar{R}(t+\delta)Z(t),Z(t)\big\rangle
  +2\big\langle\bar{N}(t+\delta)u(t),u(t)\big\rangle\\
 &\qquad\quad-2\big\langle Y(t+\delta)+\Lambda_i(t+\delta),P_i(t+\delta)\bar{B}(t+\delta)Z(t)+P_i(t+\delta)\bar{C}(t+\delta)u(t)\big\rangle\Big]dt\\
 &+\mathbb{E}\int_{s-\delta}^s\Big[2\big\langle\bar{Q}(t+\delta)\varphi(t),\varphi(t)\big\rangle+2\langle\bar{R}(t+\delta)\psi(t),\psi(t)\big\rangle
  +2\langle\bar{N}(t+\delta)\eta(t),\eta(t)\big\rangle\Big]dt.
\end{aligned}\end{equation*}
Similarly, noting $\bar{A}(t)=\bar{B}(t)=\bar{C}(t)=0$ for $t\in[s,s+\delta]$, we have
\begin{equation*}\begin{aligned}
&\mathbb{E}\int_s^T\big\langle Y(t)+\Lambda_i(t),\big[P_i(t)\bar{B}(t)\mathcal{R}_i^{-1}(t-\delta)\bar{B}(t)^\top P_i(t)
 +P_i(t)\bar{C}(t)\mathcal{N}^{-1}(t-\delta)\bar{C}(t)^\top P_i(t)\big]\\
&\qquad\times(Y(t)+\Lambda_i(t))-2P_i(t)\bar{A}(t)Y(t-\delta)-2P_i(t)\bar{B}(t)\mathcal{R}_i^{-1}(t-\delta)P_i(t-\delta)\Gamma_i(t-\delta)\\
&\qquad+2P_i(t)\bar{B}(t)\mathcal{R}_i^{-1}(t-\delta)B(t-\delta)^\top P_i(t-\delta)\Lambda_i(t-\delta)\\
&\qquad+2P_i(t)\bar{C}(t)\mathcal{N}^{-1}(t-\delta)C(t-\delta)^\top P_i(t-\delta)\Lambda_i(t-\delta)\big\rangle dt\\
&=\mathbb{E}\int_{s}^{T}\big\langle Y(t+\delta)+\Lambda_i(t+\delta),\big[P_i(t+\delta)\bar{B}(t+\delta)\mathcal{R}_i^{-1}(t)\bar{B}(t+\delta)^\top P_i(t+\delta)\\
&\qquad+P_i(t+\delta)\bar{C}(t+\delta)\mathcal{N}^{-1}(t)\bar{C}(t+\delta)^\top P_i(t+\delta)\big](Y(t+\delta)+\Lambda_i(t+\delta))\\
&\qquad-2P_i(t+\delta)\bar{A}(t+\delta)Y(t)-2P_i(t+\delta)\bar{B}(t+\delta)\mathcal{R}_i^{-1}(t)P_i(t)\Gamma_i(t)+2P_i(t+\delta)\bar{B}(t+\delta)\\
&\qquad\times\mathcal{R}_i^{-1}(t)B(t)^\top P_i(t)\Lambda_i(t)+2P_i(t+\delta)\bar{C}(t+\delta)\mathcal{N}^{-1}(t)C(t)^\top P_i(t)\Lambda_i(t)\big\rangle dt.
\end{aligned}\end{equation*}
Since (\ref{important equality}) holds, combining the above, we obtain
\begin{equation}\begin{aligned}\label{look for lower bound-1}
&2J(s,\xi;u(\cdot))=\mathbb{E}\bigg\{2\big\langle\bar{G}Y(s-\delta),Y(s-\delta)\big\rangle+2\big\langle\Lambda_i(s),(I+2\Sigma_i(s)G)^{-1}G\Lambda_i(s)\big\rangle\\
&\ +\big\langle(P_i(s)+2G)\big[Y(s)+(I+2\Sigma_i(s)G)^{-1}\Lambda_i(s)\big],Y(s)+(I+2\Sigma_i(s)G)^{-1}\Lambda_i(s)\big\rangle\\
&\ +\int_{s-\delta}^s\Big[2\big\langle\bar{Q}(t+\delta)\varphi(t),\varphi(t)\big\rangle+2\big\langle\bar{R}(t+\delta)\psi(t),\psi(t)\big\rangle
 +2\langle\bar{N}(t+\delta)\eta(t),\eta(t)\big\rangle\Big]dt\\
&\ +\int_s^T\Big[2\big\langle[Q(t)+\bar{Q}(t+\delta)]\Lambda_i(t),\Lambda_i(t)\big\rangle
 +2\big\langle[R(t)+\bar{R}(t+\delta)]\mathcal{R}_i^{-1}(t)P_i(t)\Gamma_i(t),\Gamma_i(t)\big\rangle\\
&\quad+\Big\langle\mathcal{N}(t)\Big\{u(t)-\mathcal{N}^{-1}(t)\big[C(t)^\top P_i(t)(Y(t)+\Lambda_i(t))+\bar{C}(t+\delta)^\top P_i(t+\delta)(Y(t+\delta)\\
&\quad+\Lambda_i(t+\delta))\big]\Big\},u(t)-\mathcal{N}^{-1}(t)\big[C(t)^\top P_i(t)(Y(t)+\Lambda_i(t))+\bar{C}(t+\delta)^\top P_i(t+\delta)\\
&\quad\times(Y(t+\delta)+\Lambda_i(t+\delta))\big]\Big\rangle+\Big\langle\mathcal{R}_i(t)\Big\{Z(t)-\mathcal{R}_i^{-1}(t)\big[B(t)^\top P_i(t)(Y(t)+\Lambda_i(t))\\
&\quad+\bar{B}(t+\delta)^\top P_i(t+\delta)(Y(t+\delta)+\Lambda_i(t+\delta))-P_i(t)\Gamma_i(t)\big]\Big\},Z(t)-\mathcal{R}_i^{-1}(t)\big[B(t)^\top\\
&\quad\times P_i(t)(Y(t)+\Lambda_i(t))+\bar{B}(t+\delta)^\top P_i(t+\delta)(Y(t+\delta)+\Lambda_i(t+\delta))-P_i(t)\Gamma_i(t)\big]\Big\rangle\Big]dt\bigg\}.
\end{aligned}\end{equation}
By (\textbf{A3}), $P_i(\cdot)+2G\geq 0$, $\mathcal{R}_i(\cdot)>0$, $\mathcal{N}(\cdot)>0$, it follows that
\begin{equation}\begin{aligned}\label{look for lower bound-2}
&J(s,\xi;u(\cdot))\geq\mathbb{E}\bigg\{\big\langle\bar{G}Y(s-\delta),Y(s-\delta)\big\rangle+\big\langle\Lambda_i(s),(I+2\Sigma_i(s)G)^{-1}G\Lambda_i(s)\big\rangle\\
&\ +\int_{s-\delta}^s\Big[\big\langle\bar{Q}(t+\delta)\varphi(t),\varphi(t)\big\rangle
 +\big\langle\bar{R}(t+\delta)\psi(t),\psi(t)\big\rangle+\big\langle\bar{N}(t+\delta)\eta(t),\eta(t)\big\rangle\Big]dt\\
&\ +\int_s^T\Big[\big\langle\big[R(t)+\bar{R}(t+\delta)\big]\big[2\Sigma_i(t)(R(t)+\bar{R}(t+\delta))+I\big]^{-1}\Gamma_i(t),\Gamma_i(t)\big\rangle\\
&\qquad+\big\langle(Q(t)+\bar{Q}(t+\delta))\Lambda_i(t),\Lambda_i(t)\big\rangle\Big]dt\bigg\}.
\end{aligned}\end{equation}
Thus the right hand of (\ref{look for lower bound-2}) is independent of $P_i(\cdot)$. Therefore, letting $i\rightarrow\infty$, by Proposition \ref{pro4.7} we derive
\begin{equation}\begin{aligned}\label{look for lower bound-3}
&J(s,\xi;u(\cdot))\geq\mathbb{E}\bigg\{\big\langle\bar{G}Y(s-\delta),Y(s-\delta)\big\rangle+\big\langle\Lambda(s),(I+2\Sigma(s)G)^{-1}G\Lambda(s)\big\rangle\\
&\ +\int_{s-\delta}^s\Big[\big\langle\bar{Q}(t+\delta)\varphi(t),\varphi(t)\big\rangle
 +\big\langle\bar{R}(t+\delta)\psi(t),\psi(t)\big\rangle+\big\langle\bar{N}(t+\delta)\eta(t),\eta(t)\big\rangle\Big]dt\\
&\ +\int_s^T\Big[\big\langle\big[R(t)+\bar{R}(t+\delta)\big]\big[2\Sigma(t)(R(t)+\bar{R}(t+\delta))+I\big]^{-1}\Gamma(t),\Gamma(t)\big\rangle\\
&\qquad+\big\langle(Q(t)+\bar{Q}(t+\delta))\Lambda(t),\Lambda(t)\big\rangle\Big]dt\bigg\}.
\end{aligned}\end{equation}

Now we have found the lower bound of the cost, next we will look for a control which achieves exactly this lower bound. First we introduce the stochastic Hamilonian system as follows:
\begin{equation}\left\{\begin{aligned}\label{stochastic Hamilonian-*}
 dX^*(t)=&\Big\{A(t)^\top X^*(t)-2(Q(t)+\bar{Q}(t+\delta))Y^*(t)+\mathbb{E}^{\mathcal{F}_t}\big[(\bar{A}^\prime X^*)|_{t+\delta}\big]\Big\}dt\\
         &+\Big\{B(t)^\top X^*(t)-2[R(t)+\bar{R}(t+\delta)]Z^*(t)+\mathbb{E}^{\mathcal{F}_t}\big[(\bar{B}^\top X^*)|_{t+\delta}\big]\Big\}dW(t),\\
-dY^*(t)=&\Big\{A(t)Y^*(t)+\bar{A}(t)Y^*(t-\delta)+B(t)Z^*(t)+\bar{B}(t)Z^*(t-\delta)\\
         &\ +C(t)\mathcal{N}^{-1}(t)\big\{C(t)^\top X^*(t)+\mathbb{E}^{\mathcal{F}_t}\big[(\bar{C}^\top X^*)|_{t+\delta}\big]\big\}\\
         &\ +\bar{C}(t)\mathcal{N}^{-1}(t-\delta)\big\{C(t-\delta)^\top X^*(t-\delta)+\mathbb{E}^{\mathcal{F}_{t-\delta}}\big[\bar{C}(t)^\top X^*(t)\big]\big\}\Big\}dt\\
         &-Z^*(t)dW(t),\quad t\in[s,T],\\
  X^*(s)=&-2GY^*(s),\ X^*(t)=0,\quad t\in(T,T+\delta],\\
  Y^*(T)=&\ \xi,\ Y^*(t)=\varphi(t),\ Z^*(t)=\psi(t),\quad t\in[s-\delta,s).
\end{aligned}\right.\end{equation}
\begin{mypro}\label{pro5.1}
Suppose $\bar{A}(t)=\bar{B}(t)=\bar{C}(t)=0$ for $t\in[s,s+\delta]$ and (\ref{important equality}) holds, and suppose $\Sigma(\cdot)\in C([s,T];\textbf{S}_+^n)$, $L(\cdot)\in C([s,T];\textbf{S}_+^n)$, $(\bar{X}(\cdot),\Lambda(\cdot),\Gamma(\cdot))$ $\in$ $ L^2_{\mathcal{F}}(\Omega;C([s,T];\textbf{R}^n))$ $\times L^2_{\mathcal{F}}(\Omega;C([s,T];\textbf{R}^n))\times L_{\mathcal{F}}^2([s,T];\textbf{R}^n)$ are the solutions to (\ref{delayed Riccati}), (\ref{Riccati without delay}), (\ref{linear delayed and time-advanced BSDE}), respectively. Then the stochastic Hamilton system (\ref{stochastic Hamilonian-*}) is uniquely solvable. Moreover, for $t\in[s,T]$, the following relationships hold:
\begin{eqnarray}\left\{\begin{aligned}\label{relations}
Y^*(t)=&\ \Sigma(t)X^*(t)-\Lambda(t),\\
Z^*(t)=&\ \Sigma(t)\mathcal{M}^{-1}(t)\big\{B(t)^\top X^*(t)+\mathbb{E}^{\mathcal{F}_t}\big[(\bar{B}^\top X^*)|_{t+\delta}\big]\big\}\\
      &-\big[2\Sigma(t)(R(t)+\bar{R}(t+\delta))+I\big]^{-1}\Gamma(t),\quad t\in[s,T],\\
Y^*(s)=&-(I+2\Sigma(s)G)^{-1}\Lambda(s).
\end{aligned}\right.\end{eqnarray}
\end{mypro}

\begin{proof}
For the existence of solutions to (\ref{stochastic Hamilonian-*}), define $\bar{Y}(t):=\Sigma(t)\bar{X}(t)-\Lambda(t)$ for $t\in[s,T]$, noting (\ref{important equality}) and Proposition \ref{pro4.7}, we get
\begin{equation*}\begin{aligned}
  &-\Sigma(t)\mathbb{E}^{\mathcal{F}_t}\big[(\bar{A}^\top\bar{X})|_{t+\delta}\big]
   =-\lim\limits_{i\rightarrow\infty}\Sigma_i(t)\mathbb{E}^{\mathcal{F}_t}\big[(\bar{A}^\top\bar{X})|_{t+\delta}\big]\\
  &=\lim\limits_{i\rightarrow\infty}\Big\{\Sigma_i(t)P_i(t)B(t)\mathcal{R}_i^{-1}(t)\mathbb{E}^{\mathcal{F}_t}\big[(\bar{B}^\top\bar{X})|_{t+\delta}\big]
   +\Sigma_i(t)P_i(t)C(t)\mathcal{N}^{-1}(t)\mathbb{E}^{\mathcal{F}_t}\big[(\bar{C}^\top\bar{X})|_{t+\delta}\big]\Big\}\\
  &=B(t)\Sigma(t)\mathcal{M}^{-1}(t)\mathbb{E}^{\mathcal{F}_t}\big[(\bar{B}^\top\bar{X})|_{t+\delta}\big]
   +C(t)\mathcal{N}^{-1}(t)\mathbb{E}^{\mathcal{F}_t}\big[(\bar{C}^\top\bar{X})|_{t+\delta}\big],
\end{aligned}\end{equation*}
and similarly
\begin{equation*}\begin{aligned}
  &-\bar{A}(t)\Sigma(t-\delta)\bar{X}(t-\delta)\\
  &=\bar{B}(t)\Sigma(t-\delta)\mathcal{M}^{-1}(t-\delta) B(t-\delta)^\top\bar{X}(t-\delta)+\bar{C}(t)\mathcal{N}^{-1}(t-\delta)C(t-\delta)^\top\bar{X}(t-\delta).
\end{aligned}\end{equation*}
Applying It\^o's formula, we obtain
\begin{equation*}\begin{aligned}
  -d\bar{Y}(t)=&\Big\{A(t)\bar{Y}(t)+\bar{A}(t)\bar{Y}(t-\delta)+B(t)\bar{Z}(t)+\bar{B}(t)\bar{Z}(t-\delta)\\
               &\ +C(t)\mathcal{N}^{-1}(t)\big\{C(t)^\top\bar{X}(t)+\mathbb{E}^{\mathcal{F}_t}\big[(\bar{C}^\top\bar{X})|_{t+\delta}\big]\big\}
                +\bar{C}(t)\mathcal{N}^{-1}(t-\delta)\\
               &\ \times\big\{C(t-\delta)^\top\bar{X}(t-\delta)+\mathbb{E}^{\mathcal{F}_{t-\delta}}\big[\bar{C}(t)^\top\bar{X}(t)\big]\big\}\Big\}dt
                -\bar{Z}(t)dW(t),\\
\end{aligned}\end{equation*}
where
\begin{equation}\begin{aligned}\label{bar Z}
  \bar{Z}(t)=&\ \Sigma(t)\mathcal{M}^{-1}(t)\big\{B(t)^\top\bar{X}(t)+\mathbb{E}^{\mathcal{F}_t}\big[(\bar{B}^\top\bar{X})|_{t+\delta}\big]\big\}\\
             &-\big[2\Sigma(t)(R(t)+\bar{R}(t+\delta))+I\big]^{-1}\Gamma(t).
\end{aligned}\end{equation}
Substituting $\bar{Y}(t)=\Sigma(t)\bar{X}(t)-\Lambda(t)$ for $t\in[s,T]$ and (\ref{bar Z}) into (\ref{linear delayed and time-advanced BSDE}), and letting $\bar{Y}(s)=-(I+2\Sigma(s)G)^{-1}\Lambda(s)$, it follows that
\begin{equation}\left\{\begin{aligned}\label{bar X---}
  d\bar{X}(t)=&\big\{A(t)^\top\bar{X}(t)-2(Q(t)+\bar{Q}(t+\delta))\bar{Y}(t)+\mathbb{E}^{\mathcal{F}_t}\big[(\bar{A}^\top\bar{X})|_{t+\delta}\big]\big\}dt\\
              &+\big\{B(t)^\top\bar{X}(t)-2(R(t)+\bar{R}(t+\delta))\bar{Z}(t)+\mathbb{E}^{\mathcal{F}_t}\big[(\bar{B}^\top\bar{X})|_{t+\delta}\big]\big\}dW(t),\quad t\in[s,T],\\
   \bar{X}(s)=&-2G\bar{Y}(s),\ \bar{X}(t)=0,\quad t\in(T,T+\delta].
\end{aligned}\right.\end{equation}

Let $\bar{Y}(t)=\varphi(t),\bar{Z}(t)=\psi(t)$ for $t\in[s-\delta,s)$, noting $\bar{Y}(T)=\Sigma(T)\bar{X}(T)-\Lambda(T)=\xi$, $(\bar{X}(\cdot),\bar{Y}(\cdot),\bar{Z}(\cdot))$ is a solution to the stochastic Hamiltonian system (\ref{stochastic Hamilonian-*}) and the relationships in (\ref{relations}) hold.

As for the uniqueness of solutions to (\ref{stochastic Hamilonian-*}). Suppose $(X_1(\cdot),Y_1(\cdot),Z_1(\cdot))$, $(X_2(\cdot),Y_2(\cdot),Z_2(\cdot))$ are two solutions to (\ref{stochastic Hamilonian-*}). Denote $\hat{X}(\cdot)=X_1(\cdot)-X_2(\cdot)$, $\hat{Y}(\cdot)=Y_1(\cdot)-Y_2(\cdot)$, $\hat{Z}(\cdot)=Z_1(\cdot)-Z_2(\cdot)$, then $(\hat{X}(\cdot),\hat{Y}(\cdot),\hat{Z}(\cdot))$ is a solution to the following:
\begin{equation}\left\{\begin{aligned}\label{stochastic Hamilonian---}
  d\hat{X}(t)=&\big\{A(t)^\top\hat{X}(t)-2(Q(t)+\bar{Q}(t+\delta))\hat{Y}(t)+\mathbb{E}^{\mathcal{F}_t}\big[(\bar{A}^\top\hat{X})|_{t+\delta}\big]\big\}dt\\
              &+\big\{B(t)^\top\hat{X}(t)-2(R(t)+\bar{R}(t+\delta))\hat{Z}(t)+\mathbb{E}^{\mathcal{F}_t}\big[(\bar{B}^\top\hat{X})|_{t+\delta}\big]\big\}dW(t),\\
 -d\hat{Y}(t)=&\Big\{A(t)\hat{Y}(t)+\bar{A}(t)\hat{Y}(t-\delta)+B(t)\hat{Z}(t)+\bar{B}(t)\hat{Z}(t-\delta)\\
              &\ +C(t)\mathcal{N}^{-1}(t)\big\{C(t)^\top\hat{X}(t)+\mathbb{E}^{\mathcal{F}_t}\big[(\bar{C}^\top\hat{X})|_{t+\delta}\big]\big\}\\
              &\ +\bar{C}(t)\mathcal{N}^{-1}(t-\delta)\big\{C(t-\delta)^\top\hat{X}(t-\delta)+\mathbb{E}^{\mathcal{F}_{t-\delta}}\big[\bar{C}(t)^\top\hat{X}(t)\big]\big\}\Big\}dt\\
              &-\hat{Z}(t)dW(t),\quad t\in[s,T],\\
   \hat{X}(s)=&-2G\hat{Y}(s),\ \hat{X}(t)=0,\quad t\in(T,T+\delta],\\
   \hat{Y}(T)=&\ 0,\ \hat{Y}(t)=0,\hat{Z}(t)=0,\quad t\in[s-\delta,s).
\end{aligned}\right.\end{equation}
Since $\bar{C}(t)=0$ for $t\in[s,s+\delta]$, we have
\begin{equation*}\begin{aligned}
  &\mathbb{E}\int_s^T\big\langle\hat{X}(t),C(t)\mathcal{N}^{-1}(t)\mathbb{E}^{\mathcal{F}_t}\big[(\bar{C}^\top\hat{X})|_{t+\delta}\big]\big\rangle dt\\
  &\quad=\mathbb{E}\int_{s}^{T}\big\langle\hat{X}(t-\delta),C(t-\delta)\mathcal{N}^{-1}(t-\delta)\bar{C}(t)^\top\hat{X}(t)\big\rangle dt,\\
  &\mathbb{E}\int_s^T\big\langle\hat{X}(t),\bar{A}(t)\hat{Y}(t-\delta)\big\rangle dt
   =\mathbb{E}\int_s^T\big\langle\bar{A}(t+\delta)^\top\hat{X}(t+\delta),\hat{Y}(t)\big\rangle dt,\\
  &\mathbb{E}\int_s^T\big\langle\hat{X}(t),\bar{B}(t)\hat{Z}(t-\delta)\big\rangle dt
   =\mathbb{E}\int_s^T\big\langle\bar{B}(t+\delta)^\top\hat{X}(t+\delta),\hat{Z}(t)\big\rangle dt,\\
  &\mathbb{E}\int_s^T\big\langle\hat{X}(t),\bar{C}(t)\mathcal{N}^{-1}(t-\delta)\bar{C}(t)^\top\hat{X}(t)\big\rangle dt\\
  &\quad=\mathbb{E}\int_{s}^{T}\big\langle\bar{C}(t+\delta)^\top\hat{X}(t+\delta),\mathcal{N}^{-1}(t)\bar{C}(t+\delta)^\top\hat{X}(t+\delta)\big\rangle dt.
\end{aligned}\end{equation*}
Applying It\^o's formula to $\langle\hat{X}(\cdot),\hat{Y}(\cdot)\rangle$, we derive
\begin{equation*}\begin{aligned}
  &2\langle G\hat{Y}(s),\hat{Y}(s)\rangle
=\mathbb{E}\int_s^T\Big[-2\big\langle(Q(t)+\bar{Q}(t+\delta))\hat{Y}(t),\hat{Y}(t)\big\rangle\\
&\quad-\big\langle\hat{X}(t),C(t)\mathcal{N}^{-1}(t)C(t)^\top\hat{X}(t)+\bar{C}(t)\mathcal{N}^{-1}(t-\delta)\bar{C}(t)^\top\hat{X}(t)\\
&\quad+2C(t)\mathcal{N}^{-1}(t)\bar{C}(t+\delta)^\top\hat{X}(t+\delta)\big\rangle-2\big\langle\hat{Z}(t),(R(t)+\bar{R}(t+\delta))\hat{Z}(t)\big\rangle\Big]dt\\
&=\mathbb{E}\int_s^T\Big[-2\big\langle(Q(t)+\bar{Q}(t+\delta))\hat{Y}(t),\hat{Y}(t)\big\rangle-2\big\langle\hat{Z}(t),(R(t)+\bar{R}(t+\delta))\hat{Z}(t)\big\rangle\\
&\qquad-\big\langle\mathcal{N}^{-1}(t)[C(t)^\top\hat{X}(t)+\bar{C}(t+\delta)^\top\hat{X}(t+\delta)],C(t)^\top\hat{X}(t)+\bar{C}(t+\delta)^\top\hat{X}(t+\delta)\big\rangle\Big]dt.
\end{aligned}\end{equation*}
Due to $R(\cdot)+\bar{R}(\cdot+\delta)\geq 0$, $Q(\cdot)+\bar{Q}(\cdot+\delta)\geq 0$, $G\geq 0$ and $\mathcal{N}(\cdot)>0$, it follows that
\begin{equation}\label{eq5.9}
  C(t)^\top\hat{X}(t)+\bar{C}(t+\delta)^\top\hat{X}(t+\delta)=0,\quad a.e.\ t\in[s,T],\ \mathbb{P}\mbox{-}a.s.
\end{equation}
Substituting (\ref{eq5.9}) into (\ref{stochastic Hamilonian---}), we obtain
\begin{equation}\left\{\begin{aligned}\label{eq5.10}
  -d\hat{Y}(t)=&\big[A(t)\hat{Y}(t)+\bar{A}(t)\hat{Y}(t-\delta)+B(t)\hat{Z}(t)+\bar{B}(t)\hat{Z}(t-\delta)\big]dt\\
               &-\hat{Z}(t)dW(t),\quad t\in[s,T],\\
    \hat{Y}(T)=&\ 0,\ \hat{Y}(t)=0,\ \hat{Z}(t)=0,\quad t\in[s-\delta,s).
\end{aligned}\right.\end{equation}
By the uniqueness of solutions to the linear delayed BSDE (\ref{eq5.10}), we have $(\hat{Y}(\cdot),\hat{Z}(\cdot))\equiv(0,0)$. Substituting it into (\ref{stochastic Hamilonian---}), we deduce $\hat{X}(\cdot)\equiv 0$. Hence the proof is completed.
\end{proof}

\begin{mypro}\label{pro5.2}
Let $(X^*(\cdot),Y^*(\cdot),Z^*(\cdot))$ be the solution to the stochastic Hamiltonian system (\ref{stochastic Hamilonian-*}), suppose $\bar{A}(t)=\bar{B}(t)=\bar{C}(t)=0$ for $t\in[s,s+\delta]$ and $\bar{Q}(t)=\bar{R}(t)=\bar{N}(t)=0$ for $t\in[T,T+\delta]$. Let
\begin{equation}\label{optimal control---}
  u^*(t):=\mathcal{N}^{-1}(t)\big\{C(t)^\top X^*(t)+\mathbb{E}^{\mathcal{F}_t}\big[(\bar{C}^\top X^*)|_{t+\delta}\big]\big\}.
\end{equation}
Then $(Y^*(\cdot),Z^*(\cdot))$ is the solution to the state equation (\ref{controlled delayed BSDE}) associated with the control $u^*(\cdot)$. Furthermore, the corresponding cost can be expressed as
\begin{equation}\begin{aligned}\label{optimal cost---}
  &J(s,\xi;u^*(\cdot))=\mathbb{E}\bigg\{\big\langle\bar{G}\varphi(s-\delta),\varphi(s-\delta)\big\rangle+\big\langle\Lambda(s),(I+2\Sigma(s)G)^{-1}G\Lambda(s)\big\rangle\\
  &\quad+\int_{s-\delta}^s\Big[\big\langle\bar{Q}(t+\delta)\varphi(t),\varphi(t)\big\rangle+\big\langle\bar{R}(t+\delta)\psi(t),\psi(t)\big\rangle
   +\big\langle\bar{N}(t+\delta)\eta(t),\eta(t)\big\rangle\Big]dt\\
  &\quad+\int_s^T\Big[\big\langle(R(t)+\bar{R}(t+\delta))\big[2\Sigma(t)(R(t)+\bar{R}(t+\delta))+I\big]^{-1}\Gamma(t),\Gamma(t)\big\rangle\\
  &\qquad+\big\langle(Q(t)+\bar{Q}(t+\delta))\Lambda(t),\Lambda(t)\big\rangle\Big]dt\bigg\}.
\end{aligned}\end{equation}
\end{mypro}
\begin{proof}
The first conclusion is easy to verify and we mainly prove the second conclusion. Applying It\^o's formula to $\langle X^*(\cdot),Y^*(\cdot)\rangle$, we have
\begin{equation*}\begin{aligned}
  &\frac{1}{2}\mathbb{E}\langle X^*(T),\xi\rangle+\mathbb{E}\langle GY^*(s),Y^*(s)\rangle
   =\frac{1}{2}\mathbb{E}\int_s^T\Big[-2\big\langle Y^*(t),(Q(t)+\bar{Q}(t+\delta))Y^*(t)\big\rangle\\
  &\quad-\big\langle X^*(t),C(t)\mathcal{N}^{-1}(t)C(t)^\top X^*(t)+C(t)\mathcal{N}^{-1}(t)\bar{C}(t+\delta)^\top X^*(t+\delta)\\
  &\quad+\bar{C}(t)\mathcal{N}^{-1}(t-\delta)C(t-\delta)^\top X^*(t-\delta)+\bar{C}(t)\mathcal{N}^{-1}(t-\delta)\bar{C}(t)^\top X^*(t)\big\rangle\\
  &\quad-2\big\langle Z^*(t),(R(t)+\bar{R}(t+\delta))Z^*(t)\big\rangle\Big]dt.
\end{aligned}\end{equation*}
Hence the corresponding cost functional associated with the control $u^*(\cdot)$ becomes
\begin{eqnarray*}\begin{aligned}
&J(s,\xi;u^*(\cdot))=\mathbb{E}\bigg\{\big\langle\bar{G}Y^*(s-\delta),Y^*(s-\delta)\big\rangle-\frac{1}{2}\big\langle X^*(T),\xi\big\rangle\\
&\ +\int_s^T\Big[-\big\langle Y^*(t),(Q(t)+\bar{Q}(t+\delta))Y^*(t)\big\rangle-\frac{1}{2}\big\langle X^*(t),C(t)\mathcal{N}^{-1}(t)C(t)^\top X^*(t)\\
&\qquad+\frac{1}{2}C(t)\mathcal{N}^{-1}(t)\bar{C}(t+\delta)^\top X^*(t+\delta)+\frac{1}{2}\bar{C}(t)\mathcal{N}^{-1}(t-\delta)C(t-\delta)^\top X^*(t-\delta)\\
&\qquad+\frac{1}{2}\bar{C}(t)\mathcal{N}^{-1}(t-\delta)\bar{C}(t)^\top X^*(t)\big\rangle-\big\langle Z^*(t),(R(t)+\bar{R}(t+\delta))Z^*(t)\big\rangle\\
\end{aligned}\end{eqnarray*}\begin{eqnarray*}\begin{aligned}
&\qquad+\big\langle Q(t)Y^*(t),Y^*(t)\big\rangle+\big\langle\bar{Q}(t)Y^*(t-\delta),Y^*(t-\delta)]\big\rangle+\big\langle R(t)Z^*(t),Z^*(t)\big\rangle\\
&\qquad+\big\langle\bar{R}(t)Z^*(t-\delta),Z^*(t-\delta)\big\rangle+\big\langle N(t)u^*(t),u^*(t)\big\rangle
 +\big\langle\bar{N}(t)u^*(t-\delta),u^*(t-\delta)\big\rangle\Big]dt\bigg\}.
\end{aligned}\end{eqnarray*}
Since $\bar{Q}(t)=\bar{R}(t)=\bar{N}(t)=0$ for $t\in[T,T+\delta]$, we get
\begin{equation*}\begin{aligned}
  &\mathbb{E}\int_s^T\big\langle\bar{Q}(t)Y^*(t-\delta),Y^*(t-\delta)\big\rangle dt\\
&=\mathbb{E}\int_{s-\delta}^s\big\langle\bar{Q}(t+\delta)\varphi(t),\varphi(t)\big\rangle dt+\mathbb{E}\int_s^T\big\langle\bar{Q}(t+\delta)Y^*(t),Y^*(t)\big\rangle dt,\\
&\mathbb{E}\int_s^T\big\langle\bar{R}(t)Z^*(t-\delta),Z^*(t-\delta)\big\rangle dt\\
&=\mathbb{E}\int_{s-\delta}^s\big\langle\bar{R}(t+\delta)\psi(t),\psi(t)\big\rangle dt+\mathbb{E}\int_s^T\big\langle\bar{R}(t+\delta)Z^*(t),Z^*(t)\big\rangle dt,\\
&\mathbb{E}\int_s^T\big\langle\bar{N}(t)u^*(t-\delta),u^*(t-\delta)\big\rangle dt\\
&=\mathbb{E}\int_{s-\delta}^s\big\langle\bar{N}(t+\delta)\eta(t),\eta(t)\big\rangle dt+\mathbb{E}\int_s^T\big\langle\bar{N}(t+\delta)u^*(t),u^*(t)\big\rangle dt.
\end{aligned}\end{equation*}
Noting $\bar{C}(t)=0$ for $t\in[s,s+\delta]$, we have
\begin{equation*}\begin{aligned}
  &\mathbb{E}\int_s^T\big\langle X^*(t),\bar{C}(t)\mathcal{N}^{-1}(t-\delta)C(t-\delta)^\top X^*(t-\delta)\big\rangle dt\\
&=\mathbb{E}\int_s^T\big\langle\bar{C}(t+\delta)^\top X^*(t+\delta),\mathcal{N}^{-1}(t)C(t)^\top X^*(t)\big\rangle dt.\\
&\mathbb{E}\int_s^T\big\langle X^*(t),\bar{C}(t)\mathcal{N}^{-1}(t-\delta)\bar{C}(t)^\top X^*(t)\big\rangle dt\\
&=\mathbb{E}\int_s^T\big\langle\bar{C}(t+\delta)^\top X^*(t+\delta),\mathcal{N}^{-1}(t)\bar{C}(t+\delta)^\top X^*(t+\delta)\big\rangle dt.
\end{aligned}\end{equation*}
Thus it follows that
\begin{equation}\begin{aligned}\label{eq5.13}
  J(s,\xi;u^*(\cdot))=&\ \mathbb{E}\Big\{\big\langle\bar{G}Y^*(s-\delta),Y^*(s-\delta)\big\rangle
   -\frac{1}{2}\big\langle X^*(T),\xi\big\rangle+\int_{s-\delta}^s\Big[\big\langle\bar{Q}(t+\delta)\varphi(t),\varphi(t)\big\rangle\\
  &\quad+\big\langle\bar{R}(t+\delta)\psi(t),\psi(t)\big\rangle+\big\langle\bar{N}(t+\delta)\eta(t),\eta(t)\big\rangle\Big]dt\Big\}.
\end{aligned}\end{equation}
Next applying It\^o's formula to $\langle X^*(\cdot),\Lambda(\cdot)\rangle$, we obtain
\begin{equation}\begin{aligned}\label{eq5.14}
 &-\mathbb{E}\langle X^*(T),\xi\rangle
=\mathbb{E}\langle X^*(s),\Lambda(s)\rangle+\mathbb{E}\int_s^T\Big\{\big\langle\Lambda(t),2(Q(t)+\bar{Q}(t+\delta))(-Y^*(t)\\
 &\quad+\Sigma(t)X^*(t))\big\rangle+\big\langle X^*(t),-B(t)\big[2\Sigma(t)(R(t)+\bar{R}(t+\delta))+I\big]^{-1}\Gamma(t)\\
 &\quad-\bar{B}(t)\big[2\Sigma(t-\delta)(R(t-\delta)+\bar{R}(t))+I\big]^{-1}\Gamma(t-\delta)\big\rangle\\
 &\quad+\big\langle\Gamma(t),B(t)^\top X^*(t)+\bar{B}(t+\delta)^\top X^*(t+\delta)-2(R(t)+\bar{R}(t+\delta))Z^*(t)\big\rangle\Big\}dt.
\end{aligned}\end{equation}
Substituting (\ref{relations}) and (\ref{eq5.14}) into (\ref{eq5.13}), we derive (\ref{optimal cost---}). The proof is complete.
\end{proof}

Next we give the proof of Theorem \ref{thm3.2}.
\begin{proof}
The existence and uniqueness of the solution to the stochastic Hamiltonian system (\ref{stochastic Hamiltonian}), has been proved by Proposition \ref{pro5.1}. The optimality of (\ref{optimal control-2}) can be obtained by (\ref{optimal cost---}) associated with the control (\ref{optimal control---}), which is the lower bound of the cost (see (\ref{look for lower bound-3})). Finally we are able to conclude that the control (\ref{optimal control-2}) is unique because {\bf Problem (D-BSLQ)} is a strictly convex optimization problem.
\end{proof}

Now we continue to prove Theorem \ref{thm3.1}. First we give the following proposition.

\begin{mypro}\label{pro5.3}
Let $(X^*(\cdot),Y^*(\cdot),Z^*(\cdot))$ be the solution to the stochastic Hamiltonian system (\ref{stochastic Hamiltonian}) and $L(\cdot)$, $S(\cdot)$ be the solutions to (\ref{Riccati without delay}) and (\ref{linear ASDDE}), respectively. Suppose $\bar{A}(t)=\bar{B}(t)=\bar{C}(t)=0$ for $t\in[s,s+\delta]\cup[T,T+\delta]$, then the following relationship holds:
\begin{equation}\label{relations---}
  X^*(t)=-L(t)Y^*(t)+S(t),\quad t\in[s,T],\ \mathbb{P}\mbox{-}a.s.
\end{equation}
\end{mypro}

\begin{proof}
Let us assume for the time being that $S(\cdot)$ be the solution to the following equation:
\begin{equation}\left\{\begin{aligned}\label{S}
dS(t)=&\bigg\{\Big\{A(t)^\top-L(t)B(t)\Sigma(t)\mathcal{M}^{-1}(t)B(t)^\top-L(t)C(t)\mathcal{N}^{-1}(t)C(t)^\top\Big\}S(t)\\
      &\quad+\mathbb{E}^{\mathcal{F}_t}\big[(\bar{A}^\top S)|_{t+\delta}\big]-L(t)B(t)\Sigma(t)\mathcal{M}^{-1}(t)\mathbb{E}^{\mathcal{F}_t}\big[(\bar{B}^\top S)|_{t+\delta}\big]\\
      &\quad+L(t)B(t)\big\{[2\Sigma(t)[R(t)+\bar{R}(t+\delta)]+I\big\}^{-1}\Gamma(t)\\
      &\quad-L(t)\bar{B}(t)\Sigma(t-\delta)\mathcal{M}^{-1}(t-\delta)(B^\top S)|_{t-\delta}-L(t)\bar{B}(t)\Sigma(t-\delta)\mathcal{M}^{-1}(t-\delta)\\
      &\quad\times\bar{B}(t)^\top S(t)+L(t)\bar{B}(t)\big\{2\Sigma(t-\delta)[R(t-\delta)+\bar{R}(t)]+I\big\}^{-1}\Gamma(t-\delta)\\
      &\quad-L(t)C(t)\mathcal{N}^{-1}(t)\mathbb{E}^{\mathcal{F}_{t}}\big[(\bar{C}^\top S)|_{t+\delta}\big]-L(t)\bar{C}(t)(\mathcal{N}^{-1}C^\top S)|_{t-\delta}\\
      &\quad-L(t)\bar{C}(t)\mathcal{N}^{-1}(t-\delta)\bar{C}(t)^\top S(t)+\Big\{L(t)B(t)\Sigma(t)\mathcal{M}^{-1}(t)\bar{B}(t+\delta)^\top\\
      &\quad+L(t)C(t)\mathcal{N}^{-1}(t)\bar{C}(t+\delta)^\top-\bar{A}(t+\delta)^\top\Big\}L(t+\delta)\mathbb{E}^{\mathcal{F}_t}\big[Y^*(t+\delta)\big]\\
      &\quad+L(t)\Big\{\bar{B}(t)\Sigma(t-\delta)\mathcal{M}^{-1}(t-\delta)(B^\top L)|_{t-\delta}-\bar{A}(t)\\
      &\quad+\bar{C}(t)(\mathcal{N}^{-1}C^\top L)|_{t-\delta}\Big\}Y^*(t-\delta)-L(t)\big[\bar{B}(t)\Sigma(t-\delta)\mathcal{M}^{-1}(t-\delta)\bar{B}(t)^\top\\
      &\quad+\bar{C}(t)\mathcal{N}^{-1}(t-\delta)\bar{C}(t)^\top\big]\big[\mathbb{E}^{\mathcal{F}_{t-\delta}}(\bar{X}(t))-\bar{X}(t)\big]\bigg\}dt\\
      &+\bigg\{\big[I+L(t)\Sigma(t)\big]\mathcal{M}^{-1}(t)\big\{B(t)^\top S(t)+\mathbb{E}^{\mathcal{F}_t}\big[(\bar{B}^\top S)|_{t+\delta}\big]\big\}\\
      &\qquad-\big[L(t)-2R(t)-2\bar{R}(t+\delta)\big]\big\{2\Sigma(t)[R(t)+\bar{R}(t+\delta)]+I\big\}^{-1}\Gamma(t)\\
      &\qquad-\big[I+L(t)\Sigma(t)\big]\mathcal{M}^{-1}(t)(\bar{B}^\top L)|_{t+\delta}\mathbb{E}^{\mathcal{F}_t}\big[Y^*(t+\delta)\big]\\
      &\qquad-\big[I+L(t)\Sigma(t)\big]\mathcal{M}^{-1}(t)B(t)^\top L(t)Y^*(t)\bigg\}dW(t),\quad t\in[s,T],\\
 S(s)=&\ 0,\ S(t)=0,\quad t\in[s-\delta,s)\cup(T,T+\delta].
\end{aligned}\right.\end{equation}
Note that (\ref{linear ASDDE}) and (\ref{S}) are very similar. In fact, it will be shown later that (\ref{linear ASDDE}) and (\ref{S}) have the same solutions. However, it is easier to deal with (\ref{S}).

Recall we have shown $X^*(\cdot)=\bar{X}(\cdot)$, which satisfies (\ref{linear delayed and time-advanced BSDE}). Applying It\^o's formula to $L(\cdot)Y^*(\cdot)$ and using (\ref{relations}), (\ref{linear delayed and time-advanced BSDE}), (\ref{optimal control---}), we derive
\begin{equation}\begin{aligned}\label{eq5.17}
 &d\big[-X^*(t)-L(t)Y^*(t)+S(t)\big]\\
=&\Big\{\big\{A(t)^\top-L(t)B(t)\Sigma(t)\mathcal{M}^{-1}(t)B(t)^\top-L(t)C(t)\mathcal{N}^{-1}(t)C(t)^\top\\
 &\ -L(t)\bar{B}(t)\Sigma(t-\delta)\mathcal{M}^{-1}(t-\delta)\bar{B}(t)^\top-L(t)\bar{C}(t)\mathcal{N}^{-1}(t-\delta)\bar{C}(t)^\top\big\}\\
 &\ \times\big[-X^*(t)-L(t)Y^*(t)+S(t)\big]\\
 &\ +\big\{-L(t)B(t)\Sigma(t)\mathcal{M}^{-1}(t)\bar{B}(t+\delta)^\top-L(t)C(t)\mathcal{N}^{-1}(t)\bar{C}(t+\delta)^\top+\bar{A}(t+\delta)^\top\big\}\\
 &\ \times\mathbb{E}^{\mathcal{F}_t}\big[(-X^*-LY^*+S)|_{t+\delta}\big]\\
 &\ +\big\{-L(t)\bar{B}(t)\Sigma(t-\delta)\mathcal{M}^{-1}(t-\delta)B(t-\delta)^\top-L(t)\bar{C}(t)\mathcal{N}^{-1}(t-\delta)C(t-\delta)^\top\big\}\\
 &\ \times\big[(-X^*-LY^*+S)|_{t-\delta}\big]\Big\}dt\\
 &+\Big\{(I+L(t)\Sigma(t))\mathcal{M}^{-1}(t)B(t)^\top\big[-X^*(t)-L(t)Y^*(t)+S(t)\big]\\
 &\quad+(I+L(t)\Sigma(t))\mathcal{M}^{-1}(t)\bar{B}(t+\delta)^\top\mathbb{E}^{\mathcal{F}_t}\big[(-X^*-LY^*+S)|_{t+\delta}\big]\Big\}dW(t),
\end{aligned}\end{equation}
and $-X^*(s)-L(s)Y^*(s)+S(s)=2GY^*(s)-2GY^*(s)=0$. Let $X^*(t)=L(t)=0$ for $t\in[s-\delta,s)$, then we have $-X^*(t)-L(t)Y^*(t)+S(t)=0$ for $t\in[s-\delta,s)$. Let $L(t)=0$ for $t\in(T,T+\delta]$, then it yields $-X^*(t)-L(t)Y^*(t)+S(t)=0$ for $t\in(T,T+\delta]$. Thus by the unique solvability of the ASDDE (\ref{eq5.17}), we obtain
\begin{equation}\label{relations---*}
  -X^*(t)-L(t)Y^*(t)+S(t)=0,\quad t\in[s,T],
\end{equation}
where $S(\cdot)$ is the solution to (\ref{S}). Substituting (\ref{relations---*}) into (\ref{relations}), it follows that
\begin{equation}\label{relations---**}
  Y^*(t)=(I+\Sigma(t)L(t))^{-1}(\Sigma(t)S(t)+\Lambda(t)),\quad t\in[s,T].
\end{equation}
Finally, substituting (\ref{relations---**}) into (\ref{S}), it follows that $S(\cdot)$ is the unique solution to (\ref{linear ASDDE}). The proof is complete.
\end{proof}

Finally we give the proof of Theorem \ref{thm3.1}.
\begin{proof}
It apparently follows from Proposition \ref{pro5.2} and Proposition \ref{pro5.3}.
\end{proof}

\section{Concluding Remarks}

In this paper, we have discussed the D-BSLQ optimal control problem. It is, in fact, an optimal control problem where the state equation is a controlled linear delayed BSDE and the cost is a quadratic functional. This kind of optimal control problem has three interesting characteristics worthy of being emphasized. Firstly, it is solved by using the completion-of-squares technique, where the optimal control is a linear feedback of the entire past history and the future state trajectory in a short period of time, and the optimal cost is expressed by a delayed Riccati equation and a DABSDE, which is different from the BSLQ optimal control problem without delay. Secondly, a new class of ASDDEs is introduced to seek the state feedback expression of the optimal control, which has not been studied yet. Thirdly, the existence and uniqueness of the delayed Riccati equations mentioned above are discussed in detail, which have not appeared in the previous literature.

Solvability of the DABSDE (\ref{linear delayed and time-advanced BSDE}) is an open problem. Differential games of delayed BSDEs and {\it forward-backward SDEs} (FBSDEs for short) are well worth studying, considering both Nash and Stackelberg equilibria (\cite{YJ08,WY10,WY12,SW16,XSZ18,WXX18}). The solvability of related delayed Riccati equations are more complicated, and some numerical methods are desirable. We will consider these topics in the future.


\end{document}